\documentclass {article}
\usepackage[english]{babel}
\usepackage[latin9]{inputenc}
\usepackage[T1]{fontenc}
\usepackage{graphicx}
\usepackage{hyperref}
\usepackage[final]{pdfpages}

\usepackage{mathtools}
\usepackage{amsmath,amsfonts}
\usepackage{amsthm}
\usepackage{mathtools}
\usepackage{dsfont}
\usepackage{mathenv}
\usepackage{empheq}
\usepackage{enumitem}  

{\theoremstyle{definition} \newtheorem{thm}{Theorem}[section]}
{\theoremstyle{plain} \newtheorem{prop}[thm]{Proposition}} 
\newtheorem{coro}[thm]{Corollary}

\theoremstyle{remark}
	\newtheorem{Rmq}[thm]{Remark} 
\theoremstyle{definition}
	
\theoremstyle{lemma}
	\newtheorem{Lemme}[thm]{Lemma}
	
\newtheoremstyle{mytheoremstyle} 
    {\topsep}                    
    {\topsep}                    
    {\itshape}                   
    {}                           
    {\scshape}                   
    {.}                          
    {.5em}                       
    {}  

\theoremstyle{mytheoremstyle}

\newcommand{\de}{\mathrm{d}}
\newcommand{\D}{\partial}
\newcommand{\eps}{\epsilon}
\newcommand{\R}{\mathbb{R}}
\newcommand{\p}{p_i}

\title{Dynamics of concentration in a population model structured by age and a phenotypical trait }

\author{
Samuel Nordmann\thanks{Ecole des Hautes Etudes en Sciences Sociales, CAMS,
190-198, avenue de France 75244 Paris Cedex 13, France} \footnotemark[3]
\and
Beno\^ \i t Perthame\thanks{Sorbonne Universit�s, UPMC Univ Paris 06, Laboratoire Jacques-Louis Lions  UMR CNRS 7598, UPD, Inria~de~Paris, F75005 Paris, France}  
\thanks{Emails: samuel.nordmann@ehess.fr, benoit.perthame@upmc.fr, taing@ann.jussieu.fr}
\and
C\'ecile Taing\footnotemark[2] \footnotemark[3] 
}

\mathtoolsset{showonlyrefs,showmanualtags}

\begin{document}
\maketitle

\begin{abstract}
We study a mathematical model describing the growth process of a population structured by age and a phenotypical trait, subject to aging, competition between individuals and rare mutations. Our goals are to describe the asymptotic behavior of the solution to a renewal type equation, and then to derive properties that illustrate the adaptive dynamics of such a population.
We begin with a simplified model by discarding the effect of mutations, which allows us to introduce the main ideas and state the full result. Then we discuss the general model and its limitations. 
 
Our approach uses the eigenelements of a formal limiting operator, that depend on the structuring variables of the model and define an effective fitness. Then we introduce a new method which reduces the convergence proof to entropy estimates rather than estimates on the constrained Hamilton-Jacobi equation. Numerical tests illustrate the theory and show the selection of a fittest trait according to the effective fitness.
For the problem with mutations, an unusual Hamiltonian arises with an exponential growth, for which we prove existence of a global viscosity solution, using an uncommon a priori estimate and a new uniqueness result. 

\end{abstract}

 \\[4mm]
\noindent {\bf AMS Class. No:} 	35B40, 35F21, 35Q92, 49L25. 

\section{Introduction} 
\label{sec:model}
\paragraph{} 
\label{par:}


The mathematical description of competition between populations and selection 
phenomena leads to the use of nonlocal equations that are structured by a 
quantitative trait. A mathematical way to express the selection of the fittest 
trait is to prove that the population density concentrates as a Dirac mass (or a 
sum of Dirac masses) located on this trait. This result has been obtained for 
various models with parabolic (\cite{GB.BP:07, GB.SM.BP:09, AL.SM.BP}) and 
integro-differential equations (\cite{Ba-Pe, Des-Jab-Mis-Rao, lorenzi-2013}). 
More generally, convergence to positive measures in selection-mutation models 
has  been studied by many authors, see \cite{ackleh_F_T, calsina-al-2013, 
Busse2016} for example. The question that we pose in the present paper is the 
long time behavior of the population density when the growth rate depends both 
on phenotypical fitness and age. This question brings up to consider the aging 
parameter and to use renewal type equations.
Accordingly, the aim of this paper is to study the asymptotic behavior of the solutions, as $\eps \to 0$, to the following model, with $x \geq 0$ and $y \in \R^n$:
\begin{equation}
\left\{\begin{array}{ll}
\eps \D_t m_\eps(t,x,y) + \D_x\left[A(x,y)m_\eps(t,x,y)\right] +\left(\rho_\eps(t)+d(x,y)\right)m_\eps(t,x,y)=0,
\\[2mm]
\displaystyle A(x=0,y)m_\eps\left(t,x=0,y\right)=\frac{1}{\eps^n}\int_{\R^n}\int_{\R_+}{M(\frac{y'-y}{\eps})b(x',y')m_\eps(t,x',y')dx'dy'}, 
\\[2mm]
\displaystyle\rho_\eps(t)=\int_{\R_+}\int_{\R^n}{m_\eps(t,x,y)dx dy},
\\[2mm]
m_\eps(t=0,x,y)=m_\eps^0(x,y)>0.
\end{array}\right. \label{equa}
\end{equation}

We choose $m_\eps(t,x,y)$ to represent the population density of individuals which, at time $t$, have age $x$ and trait $y$. The function $A(x,y)$ is the speed of aging for individuals with age $x$ and trait $y$. 
We denote with $\rho_\eps(t)$ the total size of the population at time $t$. Here the mortality effect features the nonlocal term $\rho_\eps(t)$, which represents competition, and an intrinsic death rate $d(x,y)>0$.
The condition at the boundary $x=0$ describes the birth of newborns that happens with rate $b(x,y)>0$  and with the probability kernel of mutation $M$. The terminology of "renewal equation" comes from this boundary condition. It is related to the McKendrick-von Foerster equation which is only structured in age (see \cite{BP} for a study of the linear equation). This model has been extended with other structuring variables as size for example (see \cite{DiekPhys, Mis-Pe-Ry}) and then with more variables (representing DNA content, maturation, etc.) to illustrate biological phenomena, among many others, like cell division (see \cite{Doumic-Gabriel, PMichel}),  proliferative and quiescent states of tumour cells (see \cite{Adimy-Cr-Ru, GyllWebb}). Space structured problems have also been extensively studied (see \cite{PEJ:RSS, Mi-patch, Mi-Pe-spatial,BP:PS-dispersal}). The variable~$x$ can represent different biological quantities that evolve throughout the individual lifespan and that are not inherited at birth. These can be as diverse as, for example,  the size of individuals, a physiological age, a parasite load and many others. 
Therefore we assume that the propgression speed $A$ depends both on $x$ and the trait $y$ to keep the model~\eqref{equa} quite general. In the present paper, we refer to $x$ as the age for simplicity. Studies in these contexts can be found in~\cite{calsina-2013}  about the existence of steady states for a selection-mutation model structured by physiological age and maturation age, which is considered as a phenotypical trait.

The parameter $\eps > 0$ is used for a time rescaling, since 
we consider selection-mutation phenomena that occur 
in a longer time scale than in an individual life cycle. 
It is also introduced to consider rare mutations. 
This rescaling is a classical way to give a continuous 
formulation of the adaptive evolution of a phenotypically 
structured population, in particular to analyze the 
dynamics of "$\bar y_\eps (t)$", the fittest trait at 
time $t$, which is solution to a form of a canonical 
equation from the framework of adaptive dynamics 
(see \cite{CFBA, OD, Di-Ja-Mi-Pe, AL.SM.BP}).\\

Here we observe two different time scales for our model. The first one is the individual life cycle time scale, i.e. the time for the population to reach the dynamical equilibrium for a fixed $y$. The second one is the evolutionary time scale, corresponding to the evolution of the population distribution with respect to the variable $y$. The mathematical expression of these two time scales is the property of variable separation
\begin{equation}
m_\eps (t,x,y) \simeq \bar \rho (t) Q(x,y) \delta_{y =\bar y (t)},\nonumber
\end{equation}
when $\eps$ is close to $0$, where $Q(x,y)$ is a normalized equilibrium distribution over age for a fixed $y$, $\bar \rho (t)$ the total population density and $\bar y (t)$ the fittest trait at the limit $\eps \to 0$.
In order to observe the asymptotic behavior of the solution to \eqref{equa}, the key point is to prove convergence results when $\eps$ vanishes, that is when the two time scales become totally separated. In other words, as $\eps$ vanishes, we observe the ecological equilibrium and we focus on the evolutionary dynamics of the population density to identify $\bar y (t)$.
\paragraph{} 
\label{par:}

As a first step, we ignore mutations, i.e. we take $M(z)=\delta_0(z)$. Equation \eqref{equa} becomes, for $t, x \geq 0$ and $y \in \R^n$,
\begin{equation}\label{systemesansmutation}
\left\{\begin{array}{ll}
\eps \D_t m_\eps(t,x,y) + \D_x\left[A(x,y)m_\eps(t,x,y)\right] +\left(\rho_\eps(t)+d(x,y)\right)m_\eps(t,x,y)=0,
\\[2mm]
\displaystyle A(x=0,y)m_\eps\left(t,x=0,y\right)=\int_{\R_+}{b(x',y)m_\eps(t,x',y)d x'}, 
\\[2mm]
\displaystyle \rho_\eps(t)=\int_{\R^n}\int_{\R_+}{m_\eps(t,x,y)d x dy},
\\[2mm]
m_\eps(t=0,x,y)=m_\eps^0(x,y)>0.
\end{array}\right.
\end{equation}
The analysis of this simplified model allows us to introduce the main ideas of our work. In order to study the asymptotic behavior of the solution to \eqref{systemesansmutation}, we consider the associated eigenproblem, that is to find, for each $y \in \R^n$, the solution $(\Lambda(y), Q(x,y))$ to
\begin{equation}\label{valeurpropreS}
\left\{\begin{array}{ll}
\D_x\left[A(x,y)Q(x,y)\right]+d(x,y)Q(x,y)=\Lambda(y)Q(x,y),\\[2mm]
\displaystyle A(x=0,y)Q(x=0,y)=\int_{\R_+}{b(x',y)Q(x',y) dx'},
\\[2mm]
\displaystyle Q(x,y)> 0, \quad \int_{\mathbb{R}^+}b(x',y)Q(x',y)d x'=1.
\end{array}\right.
\end{equation}
We also define $\Phi$, solution of the dual problem 
\begin{equation}\label{equaduale2S}
\left\{\begin{array}{l}
A(x,y)\D_x\Phi(x,y)+\left[\Lambda(y)-d(x,y)\right]\Phi(x,y)=-b(x,y)\Phi(0,y),\\[2mm]
\displaystyle \int_{\R_+}{Q(x,y)\Phi(x,y)d x}=1.
\end{array}\right.
\end{equation}

\paragraph{} 
\label{par:}

The purpose of this paper is to introduce an alternative to the usual WKB method (see \cite{GB.BP:07, Di-Ja-Mi-Pe}) to prove the concentration phenomenon in the $y$ variable for the model \eqref{systemesansmutation}. Indeed we propose a new approach that consists in firstly introducing the exponential concentration singularity and secondly in estimating the corresponding age profile. The main idea is to define a function $u_\eps(t,y)$ independent of $x$, and an "age profile" $p_\eps(t,x,y)$, such that we can write ${m_\eps(t,x,y)=p_\eps(t,x,y)\mathrm{e}^{\frac{u_\eps(t,y)}{\eps}}}$. Then we prove that $u_\eps$ converges uniformly to a function $u$, which zeros correspond to the potential concentration points of the population density when $\eps$ vanishes. Moreover, following earlier works, we prove that  $p_\eps(t,x,y)$ converges to the first eigenvector of the stationary problem introduced in~\eqref{valeurpropreS} using the general relative entropy (GRE) principle
(see \cite{PMSMBP1} for an introduction).

This convergence result does not apply for the model \eqref{equa} with mutations. Because of several technical obstructions we cannot prove the full result. However, we are able to derive some estimates resulting from the study of the formal limiting problem. Then we derive an approximation problem with a Hamilton-Jacobi equation satisfied by a sequence $u_\eps$ that we build and we prove its convergence to the solution to the constrained Hamilton-Jacobi equation coming from the formal limiting problem. This constrained Hamilton-Jacobi formally determines the locations of the concentration points. 

Recently, the asymptotic behavior of an age-structured equation with spatial 
jumps has been determined in  \cite{VC:PG:AMG} when the death rate vanishes and 
with a slowly decaying birth rate $b$; then the 
eigenproblem~\eqref{valeurpropreS} does not have a solution. Also in 
\cite{Dji-Duc-Fab}, a concentration result has been proved for a model 
representing the evolutionary epidemiology of spore producing plant pathogens in 
a host population, with infection age and pathogen strain structures.

More generally, the  Hamilton-Jacobi approach to prove the concentration of the population density goes back to \cite{Di-Ja-Mi-Pe} and has been extensively used in works on the similar issue (see \cite{champagnat2011} for example). It also has been used in the context of front propagation theory for reaction-diffusion equations (see \cite{BG-LCE-PES, combustion, Eva}). For example in the case of the simple Fisher-KPP equation, the dynamics of the front are described by the level set of a solution of a Hamilton-Jacobi equation. In this framework, it is naturally appropriate to use the theory of viscosity solutions to derive the convergence of the sequence $u_\eps$ (see \cite{bardi-capuzzo, GB:94, fle-soner} for an introduction to this notion). In this paper we also prove a uniqueness result in the viscosity sense that is not standard because the Hamiltonian under investigation has exponential growth.

\paragraph{} 
\label{par:}
The paper is organized as follows. We first state the general assumptions in section~\ref{SecAssumption}.
Section~\ref{secCaseWM} is devoted to the formulation and the proof of the convergence results in the case without mutation. In section~\ref{secCaseM}, we discuss the case with mutations and tackle the formal limit of the stationary problem. Finally we present some numerics in section~\ref{numerics}.

\section{Assumptions}\label{SecAssumption}
 
Since the analysis requires several technical assumptions on the coefficients and the initial data, we present them first.
\newline

\noindent \textit{Regularity of the coefficients.}
 We assume that $x\mapsto b(x,y)>0$ and ${x\mapsto d(x,y)>0}$ are uniformly continuous, that $x\mapsto A(x,y)$ is $\mathcal{C}^1$ and such that, for all $y \in \R^n$,
\begin{equation}\label{hypothesed}
\lim\limits_{\substack{x\to+\infty}}d(x,y)= +\infty,
\end{equation}
\vspace{-2mm}
\begin{equation}\label{hypM2}
0<\underline{r}\leq b(x,y)-d(x,y)\leq \overline{r},
\end{equation}
\vspace{-2mm}
\begin{equation}\label{bornesA}
	0<A_0\leq A(x,y)\leq A_\infty,\quad\text{for two positive constants $A_0$ and $A_\infty$}. 
\end{equation}
This set of assumptions is an example. It serves mostly to guarantee some properties of the spectral problem which are stated in Theorem~\ref{theoremeeigenS}. Only the conclusions of Theorem~\ref{theoremeeigenS} are used in the present approach to the concentration phenomena.
\\

\noindent \textit{Conditions on the initial data.} 
We suppose that the total density is initially bounded
\begin{equation}
0<\underline{\rho}^0\leq\rho_\eps^0\leq\overline{\rho}^0,\label{hypM1bis}
\end{equation}
with $\underline{\rho}_0$ and $\overline{\rho}_0$ two constants.
Besides we assume the population to be well prepared for concentration, that is, we can write 
\begin{equation}\label{hypinit}
m_\eps^0(x,y)=p_\eps^0(x,y)e^{\frac{u_\eps^0(y)}{\eps}},
\end{equation} where $u_\eps^0$ is uniformly Lipschitz continuous and
\begin{equation}\label{initialu}
	\left\{\begin{aligned}
	& \exists k_0>0, \forall \eps>0, \forall (y,y') \in \R^{2n}, \vert u^0_\eps(y) - u^0_\eps(y') \vert \leq k_0 \vert  y-y'\vert, \\[2mm]
		&u^0_\eps(y)\to u^0(y)\leq0 \text{  uniformly in $y$},\\[2mm]
		&\exists!\ \bar y^0\in\R^n, \max_{y \in \R^n} u^0(y)=u^0(\bar y^0)=0,\\
		&e^{\frac{u^0_\eps}{\eps}} \xrightharpoonup[\eps \to 0]{} \delta_{\bar y^0}.
	\end{aligned}\right.
\end{equation}
Finally, we assume that, for all $y \in \R^d$, there exist $\underline{\gamma}(y)$, $\overline{\gamma}(y)$ and $\gamma^0(y)$ positive such that, for all $\eps >0, x \in \R_+$, 
\vspace{2mm}
\begin{equation}\label{initialp}
	\underline{\gamma}(y)Q(x,y)\leq p_\eps^0(x,y)\leq \overline{\gamma}(y) Q(x,y),
\end{equation}
\vspace{-2mm}
\begin{equation}\label{hypp}
 	\int_{\R_+} \left\vert p_\eps^0(x,y)-\gamma^0(y)Q(x,y)\right\vert \Phi(x,y)dx \underset{\eps \to 0}{\longrightarrow} 0, \quad    \text{uniformly in $y$},
\end{equation} where $Q,\Phi$ are eigenelements associated with the eigenproblem~\eqref{valeurpropreS}-\eqref{equaduale2S} which properties are analyzed in section~\ref{subsec:eigenS}.
\\

\noindent\textit{Some notations}: We define, for $x \in \R_+,$ $ y \in \R^n$ and $\lambda \in \R$, the functions
\begin{equation}\label{definitionF}
f(x,y,\lambda)=\frac{b(x,y)}{A(x,y)}\mathrm{exp}\left(-\int_0^{x}{\frac{d(x',y)-\lambda}{A(x',y)}d x'}\right), \quad F(y, \lambda)= \int_{\R_+} f(x,y,\lambda)dx.
\end{equation}

\section{Case without mutations} \label{secCaseWM}
We present our new approach to understand how solutions of~\eqref{systemesansmutation} behave when $\eps$ vanishes. To prove that a concentration in the $y$ variable may occur, we first consider the principal eigenvalue $\Lambda(y)$ of~(\ref{valeurpropreS}), and define $u_\eps$ as the solution of the equation
\begin{equation}\label{hamiltonienS}
	\left\{\begin{aligned}
		&\D_t u_\eps(t,y)=-\Lambda(y)-\rho_\eps(t), \quad t>0, y \in \R^n,\\
		&u_\eps(0,y)=u_\eps^0,\quad y \in \R^n.
	\end{aligned}\right.
\end{equation}
Then, we define $p_\eps$ such that 
\begin{equation}\label{factorization}
{m_\eps(t,x,y)=p_\eps(t,x,y)e^{\frac{u_\eps(t,y)}{\eps}}},
\end{equation}
and we prove that $p_\eps$ converges when $\eps\to0$ respectively to the eigenvector $Q$ associated to $\Lambda$ in some way that we will specify, using an entropy method. Thereafter we prove that $u_\eps$ converges locally uniformly as $\eps$ goes to 0. This section is devoted to the proof of the following theorem, which states the concentration of the population density on the fittest traits.

\begin{thm}\label{Theorem1}
Assume \eqref{hypothesed}--\eqref{hypp}. Let $m_\eps$ be the solution of~\eqref{systemesansmutation}, $u_\eps$ the solution of~\eqref{hamiltonienS}, $p_\eps$ defined by the factorization~\eqref{factorization} and $(\Lambda,Q)$ defined in~\eqref{valeurpropreS}. Then, the following assertions hold true:
\begin{enumerate}[label=(\roman*)]
\item $\rho_\eps(t)=\int_{\R^n}\int_{\R_+}m_\eps(t,x,y)dxdy$ converges to a function $\rho$ when $\eps$ vanishes in $L^{\infty}(0, \infty)$ weak-$\star$.
\item $p_\eps$ converges to a multiple of the normalized eigenvector $Q$ for a weighted $L^1$ norm.
\item $u_\eps$ converges locally uniformly when $\eps$ vanishes to a continuous function $u$ solution of
\begin{equation}\label{HJSlimit}
	\left\{\begin{aligned}
		&\D_t u(t,y)=-\Lambda(y)- \rho(t), \quad t>0, y \in \R^n,\\[2mm]
		&\sup_{y \in \R^n} u(t,y) =0, \quad \forall t>0,\\
		&u(0,y)=u^0(y),\quad y\in\R^n.
	\end{aligned}\right.
\end{equation}

\item Hence, $m_\eps$ converges weakly as $\eps$ vanishes to a measure $\mu$ which support is included in $\{(t,y)\in(0,\infty)\times\mathbb{R}^n | u(t,y)=0\}$.
\item Furthermore, assuming $u^0$ and $-\Lambda$ to be strictly concave
$$m_\eps(t,x,y)\underset{\eps\to0}{\rightharpoonup}\rho(t)\frac{Q(x,y)}{\Vert Q(\cdot,y)\Vert_{L^1}}\delta_{y=\bar y(t)},$$ where $\bar y(t)\in\R^n$ satisfies a canonical differential equation.
\end{enumerate}
\end{thm}

\subsection{The eigenproblem} 
\label{subsec:eigenS}

We first study the eigenproblem \eqref{valeurpropreS} and the associated dual problem \eqref{equaduale2S}.
The operator in~\eqref{valeurpropreS}, which is time independent, is obtained by formally taking $\eps=0$ in system~\eqref{systemesansmutation} and by removing the formal limiting term $\rho(t)$. We point out that this approach relies on the observation that $\rho_\eps(t)$ operates linearly on $m_\eps$, therefore its effect on the eigenvalue $\Lambda$ is no more than a shift. 
The following theorem states existence and uniqueness for these eigenelements as well as some properties.
\begin{thm}\label{theoremeeigenS} 
We assume (\ref{hypothesed})--\eqref{bornesA}. For a given $y\in\mathbb{R}^n$, there exists a unique triplet\break $(\Lambda(y), Q(x,y), \Phi(x,y))$ solution of (\ref{valeurpropreS})-(\ref{equaduale2S}). Moreover, the function $x\mapsto Q(x,y)$ is bounded and belongs to $L^1(0,\infty)$, the function $y \mapsto \Lambda(y)$ is $\mathcal{C}^1$ and we have
\begin{equation}\label{derlambdaS}
 	\D_\lambda F>0,\quad F(y,\Lambda(y))=1,
\end{equation}
\begin{equation}\label{derlambdaS2}
\nabla_y\Lambda(y)=-\frac{\nabla_yF(y,\Lambda(y))}{\D_\lambda F(y,\Lambda(y))}, \quad \underline{r}\leq -\Lambda(y)\leq \overline{r},
\end{equation}
where $F$ is defined in \eqref{definitionF}.
\end{thm}

  \noindent The complete proof, which only uses classical arguments, is postponed to Appendix~\ref{eigenappendix}. We give here a formal idea of the method. The eigenfunction $Q$ satisfies a linear differential equation that allows us to derive
\begin{equation}
\label{expQ}
Q(x,y)=\frac{1}{A(x,y)}\mathrm{exp}\left(-\int_0^{x}{\frac{d(x',y)-\Lambda(y)}{A(x',y)}d x'}\right).
\end{equation}
From this formulation, we deduce that the eigenvalue $\Lambda(y)$ must satisfy $F(y, \Lambda(y))=1$, for all $y \in \R^d$, where $F$ is defined in~\eqref{definitionF}. Since $\D_\lambda F>0$, the above equality determines a unique $\Lambda$, and therefore a unique $Q$. Similarly, we derive an explicit formula for $\Phi$.

Note that $Q$ represents the age distribution at equilibrium for a fixed $y$, thus it seems natural that it exponentially decreases. The eigenvalue 
$\Lambda$ defines what we call the "effective fitness". It drives the adaptive dynamics of the population, as discussed in what follows.

\subsection{Concentration} 
\subsubsection{Saturation of the population density} 

The nonlocal term $\rho_\eps$ in \eqref{systemesansmutation}, which is also called competition term, can be interpreted as the pressure exerted by the total population on the survival of individuals with trait $y$. It leads the total population to be bounded. This saturation property also holds for the general model with mutations and is stated in its general form in Proposition~\ref{saturation2}.

\begin{prop}\label{saturation}
We assume \eqref{hypothesed}-\eqref{hypM1bis} and \eqref{initialp}, then, 
\begin{equation}\label{rhoM}
\forall t \geq 0, \quad  \rho_m\leq\rho_\eps(t)\leq\rho_M,\end{equation}
where $\rho_m:=\min (\underline{r}, \underline \rho^0)$ and $\rho_M:=\max (\bar r, \bar \rho^0)$.
Hence, after extraction of a subsequence, $\rho_\eps$ converges weakly-$\star$ to a function $\rho$ in $L^\infty(0,+\infty)$. 
\end{prop}
The proof of this result, using classical arguments, is postponed to Appendix~A and is given as a particular case of Proposition~\ref{saturation2}. 

Thereafter, in order to remove the restriction to a subsequence, we need a uniqueness statement to prove the assertion (i) of Theorem~\ref{Theorem1}. This is done in Section~\ref{SectionA}.

We now introduce $u_\eps$ solution to~\eqref{hamiltonienS}, and we define $p_\eps(t,x,y)$ by the factorization~\eqref{factorization} that we recall
\begin{equation*}
{m_\eps(t,x,y)=p_\eps(t,x,y)e^{\frac{u_\eps(t,y)}{\eps}}}.
\end{equation*}
We first prove the convergence of $p_\eps$. This convergence result is needed to prove the convergence of $u_\eps$ and then the uniqueness of $\rho$ and $u$.

\subsubsection{Convergence of $p_\eps$} 

We state the following theorem on the convergence of $p_\eps$, which details the statement (ii) of Theorem~\ref{Theorem1}.
\begin{thm}\label{thmCV} We assume \eqref{hypothesed}--\eqref{hypp}. With the constants defined in \eqref{initialp}--\eqref{hypp} and $(Q,\Phi)$ defined in Theorem~\ref{theoremeeigenS},
\begin{enumerate}[label=(\roman*)]
\item we have $\underline\gamma(y)Q(x,y)\leq p_\eps(t,x,y)\leq \overline\gamma(y)Q(x,y)$ for all $t\geq0$,
\item moreover, the profile $p_\eps$ converges to the eigenfunction $Q$ for a weighted $L^1$ norm. Namely, for $\gamma^0$ defined in assumption \eqref{hypp} we have, uniformly in $(t,y)$,
 \begin{equation*}\int_{\R^+}{\left\vert \frac{p_\eps}{Q}(t,x,y)-\gamma^0(y)\right\vert Q(x,y)\Phi(x,y)}dx \to 0 \    \      \text{ when $\eps\to0$},
\end{equation*}
\end{enumerate}
\end{thm}
The main ingredients of the proof are as follows: in a first step we prove that $\frac{p_\eps}{Q}$ is bounded. Then we use an entropy method to prove that the convergence occurs in a weighted $L^1$ space. Our approach follows closely \cite{PMSMBP1,BP}.
\\

\begin{proof}[Proof of Theorem \ref{thmCV}]
\textit{First step: bounds on $\frac{p_\eps}{Q}$.} From \eqref{systemesansmutation} and  \eqref{hamiltonienS}-\eqref{factorization}, we infer that $p_\eps$ satisfies
\begin{equation}\label{equaP}
\left\{\begin{array}{r l}
&\eps \D_tp_\eps(t,x,y)+\D_x\left[A(x,y)p_\eps(t,x,y)\right]+\left[d(x,y)-\Lambda(y)\right]p_\eps(t,x,y)=0,\\
&\displaystyle A(x=0,y)p_\eps(t,x=0,y)=\int_{\R^+}{b(x',y)p_\eps(t,x',y)d x'}. \\
\end{array}\right.
\end{equation}
Moreover $Q$ satisfies the same linear equation. Assumption~\eqref{initialp} and the comparison principle for transport equations prove the first statement of Theorem~\ref{thmCV}.
\\

\noindent \textit{Second step: Entropy inequality}. In the sequel, we consider
\begin{equation}\label{defqandv}
v_\eps(t,x,y):=\frac{p_\eps(t,x,y)}{Q(x,y)}-\gamma^0 (y).
\end{equation} 
We also define, for any function $f(t,x,y)$, the average
\begin{equation*}
\langle f \rangle(t,y):=\int_{\R_+}{f(t,x,y) b(x,y)Q(x,y) dx},
\end{equation*}
and we notice that a direct computation gives
	\begin{equation}\label{vectnorm}
		\left\{\begin{aligned}
	&\eps\D_t v_\eps(t,x,y)+A(x,y)\D_x v_\eps (t,x,y)=0,\\
	&v_\eps(t,x=0,y)=\langle v_\eps \rangle(t,y).
		\end{aligned}\right.
	\end{equation}
Thus we have, in distribution sense
\begin{equation}\label{vectnorm2}
\eps\D_t \vert v_\eps(t,x,y)\vert +A(x,y)\D_x \vert v_\eps (t,x,y)\vert=0.
\end{equation}
We now introduce the generalized relative entropy 
$$E_\eps(t,y)=\int_{\R_+}\vert v_\eps(t,x,y)\vert Q(x,y)\Phi(x,y) dx$$ 
and compute 
\begin{equation*}
\begin{aligned}
\eps\D_tE_\eps(t,y) & = \int_{\R_+}{ \eps\vert\D_tv_\eps(t,x,y)\vert Q(x,y)\Phi(x,y) }dx\\
& = -\int_{\R_+}{A(x,y)\vert \D_xv_\eps(t,x,y)\vert Q(x,y)\Phi(x,y) }dx\\
& = -\left[\vert v_\eps\vert A Q\Phi\right]^\infty_{x=0}+\int_{\R_+}{\vert v_\eps\vert\D_x\left( AQ\Phi \right)}dx.
\end{aligned}
\end{equation*}
The function $|v_\eps| AQ\Phi$ converges to 0 when $x$ goes to infinity,. Indeed, $v_\eps$ is bounded from the assertion (i) of Theorem~\ref{thmCV}, $A$ is bounded and, since an explicit computation of $Q\Phi$ gives
\begin{equation}
Q(x,y)\Phi(x,y)=\frac{\Phi(0,y)}{A(x,y)}\left(1-\int_0^x\frac{b(x',y)}{A(x',y)}\exp\left(\int_0^{x'}\frac{\Lambda(y)-d(x'',y)}{A(x'',y)}d x''\right)dx'\right),
\end{equation}
from the equality $F(y,\Lambda(y))=1$ in~\eqref{derlambdaS}, we deduce that $Q\Phi$ goes to 0 as $x \to \infty$. Then,
\begin{equation*}
\eps\D_tE_\eps(t,y) = \Phi(0,y)\left\vert\langle v_\eps\rangle\right\vert(t,y)-\Phi(0,y)\int_{\R_+}{bQ\vert v_\eps\vert dx}.
\end{equation*}
Hence, using the Cauchy-Schwarz inequality,
\begin{equation}\label{equaproof}
\eps\D_tE_\eps(t,y) =-{\Phi(0,y)}\left(\langle \vert v_\eps\vert\rangle-\left\vert\langle v_\eps \rangle\right\vert\right) \leq 0.
\end{equation}
Therefore $0\leq E_\eps(t,y)\leq E_\eps(0,y),$ and we conclude for (ii) using~\eqref{hypp}.
\end{proof}

\begin{Rmq}
As $v_\eps$ is bounded, the convergence stated in (iii) occurs in all weighted $L^p$ norms. Namely, for all $p\geq1$
	\begin{equation*}
		\int_{\R_+}\left\vert \frac{p_\eps}{Q}(t,x,y)-\gamma^0(y)\right\vert^p Q\Phi dx\longrightarrow0,\quad\text{when }\eps\to0.
	\end{equation*} 
\end{Rmq}

\subsubsection{Convergence of $u_\eps$}\label{SectionA}
Integrating (\ref{hamiltonienS}), we obtain the explicit formula
\begin{equation}\label{equation_ue}
u_\eps(t,y)=u_\eps^0(y)-t\Lambda(y)-\int_0^t\rho_\eps (s) ds.
\end{equation}
Hence, by \eqref{initialu} and Proposition \ref{saturation}, after extraction of a subsequence, $u_\eps$ converges locally uniformly to a function $u$ which is given by
\begin{equation}\label{equation_u}
u(t,y)=u^0(y)-t\Lambda(y)-\int_0^t\rho (s) ds.
\end{equation}
Next, we claim that 
\begin{equation}\label{constraint}
\sup\limits_{y\in\R^n} u(t,y) =0,\quad\forall t\geq 0.
\end{equation}Indeed, we recall $m_\eps(t,x,y)=p_\eps(t,x,y)e^{\frac{u_\eps(t,y)}{\eps}}$ and $p_\eps(t,x,y)$ converges in virtue of Theorem~\ref{thmCV}. If there existed a point $y_0$ for some $t$ such that $u(t,y_0)>0$, $\rho_\eps(t)$ would diverge, which is a contradiction with Proposition~\ref{saturation}. In a similar way, $\sup_y u(t,\cdot)<0$ would imply $\rho_\eps(t)\to0$, which also contradicts Proposition~\ref{saturation}. Hence \eqref{constraint} must hold.

Thus, up to extraction of a subsequence, $m_\eps$ weakly converges to a measure which support is included in the set $\{(t,y)\in[0,+\infty)\times\R^n | u(t,y)=0\}$.
Outside of this set, we know that the population density vanishes locally uniformly as $\eps\to0$.

Finally we prove the convergence of the whole sequence $u_\eps$. From \eqref{equation_u} and \eqref{constraint} we obtain
\begin{equation}\label{max_point}
\int_0^t  \rho (s) ds = \sup_{y \in \R^n} [u^0(y)-t\Lambda(y)],\quad \forall t \geq 0.
\end{equation}
The uniqueness of the limit function $\rho$ is therefore ensured, which implies that the full sequence $\rho_\eps$ converges to $\rho$. Then, the convergence of the full family $u_\eps$ follows from \eqref{equation_ue}. Hence the statements (i),(iii) and (iv) of Theorem~\ref{Theorem1}.

\subsection{Properties of concentration points}

Since we can explicitly integrate \eqref{hamiltonienS} to obtain \eqref{equation_u}, we are able to identify the points where the population concentrates, which are the points where $u$ vanishes.

\begin{prop}
	Let $t\in (0, \infty)$ and $\bar y(t)\in\R^n$ such that $u(t,\bar y(t))=0$, where $u$ is given in \eqref{equation_u}. As $\bar y(t)$ is a maximum point of $u(t,\cdot)$, it satisfies
\begin{equation}\label{PointsConcentration}
\nabla_y u^0(\bar y(t))= t \nabla_y \Lambda \left(\bar y(t)\right),
\end{equation} 
and we have 
	\begin{equation}\label{eqmax}
		u^0(\bar y(t))=\int_0^t\rho(t')dt'-t\rho(t).
	\end{equation}
\end{prop}

\begin{proof}
	From equation~\eqref{equation_u} we derive
	\begin{equation}\label{resultat1}
		u^0(\bar y(t))=t\Lambda(\bar y(t))+\int_0^t\rho(t')d t'.
	\end{equation}
Besides, $\bar y(t)$ is a maximum point of $u(t,\cdot)$, therefore $\nabla_y u(t,\bar y(t))=0$ which proves \eqref{PointsConcentration}. Moreover $\D_t u(t,\bar y(t))=0$, and using (\ref{hamiltonienS}) we obtain 
	\begin{equation}\label{resultat2}
		\Lambda(\bar y(t))=-\rho(t).
	\end{equation}
Thus, combining \eqref{resultat1} and \eqref{resultat2}, we obtain equation \eqref{eqmax}
\end{proof}
At this stage, the concentration of the population density on a single trait $\bar y(t)$ cannot be concluded yet because the above relation defines a hypersurface. There are two frameworks in which one can prove that the population is monomorphic, that is, the population converges in measure toward a Dirac mass located on a unique point $\bar y(t)$ at each time $t\geq0$. The first framework assumes that $y$ is one dimensional, and $y\mapsto\Lambda(y)$ is strictly monotonic. The second assumes, for $y\in\R^d$, that $u_\eps^0(\cdot)$ and $-\Lambda(\cdot)$ are strictly concave uniformly in $\eps$. The interested reader can refer to \cite{GB.BP:07} and \cite{AL.SM.BP} for a complete analysis of these two cases.

In the framework of uniform strict concavity, we obtain the additional result of uniform regularity on $u_\eps$ and $u$, which enables to rigorously derive a form of canonical equation in the language of adaptive dynamics. This canonical equation gives the dynamics of the selected trait, that is, the evolution of the concentration point in an evolutionary time scale.

\begin{thm}\label{thmED}
Assume that $u^0$ and $-\Lambda$ are strictly concave in a neighborhood of $\bar y^0$ defined in~\eqref{initialu}. Then $u(t,\cdot)$, given in \eqref{equation_u}, is locally strictly concave and there exists $T>0$ such that for all $ t\in(0,T)$,  $u(t,\cdot)$  reaches its maximum $0$ on a unique point $\bar y(t)$. Moreover $t \mapsto \bar y(t)\in\mathcal{C}^1\left(0,T\right)$ and its dynamics is described by the equation
\begin{equation}\label{dynamique}
\dot{\bar{y}}(t)=\left(\nabla_y^2u(t,\bar{y}(t))\right)^{-1}\cdot\nabla_y\Lambda\left(\bar{y}(t)\right), \quad \bar y (0) = \bar y^0.
\end{equation}
\end{thm}

\begin{proof}
We are interested in the solutions $\bar y(t)\in\R^n$ of
\begin{equation}\label{EquationBut}
\nabla_y u(t,\bar y(t))=0.
\end{equation}
Note that $u$ is strictly concave, because $u^0$ and $-\Lambda$ are. Therefore, such a $\bar y(t)$ must satisfy $u(t,\bar y(t))=\max_y u(t,y)=0$.

From \eqref{initialu} we know that at initial time there exists a unique solution $\bar y^0$ of \eqref{EquationBut}. Besides, as $u$ is strictly concave, $\nabla_y^2 u$ is invertible.
Hence, thanks to the implicit functions theorem, there exists $T>0$ such that  for all $t\in(0,T)$, there exists a unique $\bar y(t)\in\R^n$ satisfying \eqref{EquationBut}. Moreover, $t\mapsto \bar y (t)$ is a $\mathcal{C}^1$ function, and then differentiating \eqref{EquationBut} with respect to $t$, we obtain, using (\ref{equation_u}),
 \begin{equation*}
 	0=\frac{\mathrm{d}}{\mathrm{d}t}\left[\nabla_y u(t,\bar y (t))\right]=-\nabla_y\Lambda(\bar y (t))+\left(\nabla_y^2 u(t,\bar y (t)\right)\cdot \dot{\bar y}(t),
 \end{equation*}
and \eqref{dynamique} follows.
\end{proof}

\begin{Rmq}
Note that we have
\begin{equation}\label{diffeq_lambda}
\frac{\de}{\de t}\left[\Lambda\left(\bar y(t)\right)\right]=\left(\nabla_y\Lambda(\bar{y}(t))\right)\cdot\left(\nabla_y^2u(t,\bar{y}(t))\right)^{-1} \cdot \left(\nabla_y\Lambda(\bar{y}(t))\right).
\end{equation}
Then, we deduce that $\frac{\de}{\de t}\left[\Lambda\left(\bar y(t)\right)\right]\leq 0$. Therefore, if at initial time $\bar y^0$ belongs to a potential well of $\Lambda$, then $\bar y (t)$ remains bounded. Thus Theorem \ref{thmED} holds globally in time and $\bar y (t)$ converges to a local minimum of $\Lambda$ when $t$ goes to infinity.
\end{Rmq}

From Theorem \ref{thmED} we infer the statement (v) of Theorem~\ref{Theorem1}. We also give the following additional results. The first one is derived directly from \eqref{derlambdaS2}, the second one from \eqref{resultat2} and \eqref{diffeq_lambda}.
\begin{coro}
	Under the same hypothesis as in Theorem \ref{thmED}, the critical points for evolutionary dynamics satisfy $\nabla_y F(y^*,\Lambda(y^*))=0$.
\end{coro}

\begin{coro}\label{CoroNonExtinction}
	Under the same hypothesis as in Theorem \ref{thmED}, we have $t\mapsto\rho(t)\in\mathcal{C}^1(0,T)$ and  $\dot\rho(t)\geq0$ for all $t\in(0,T)$.
\end{coro}

\section{Case with mutations} \label{secCaseM}

We turn to the model \eqref{equa} including mutations. We use the same approach 
as in the previous section, that is, we write 
${m_\eps(t,x,y)=p_\eps(t,x,y)e^{\frac{u_\eps(t,y)}{\eps}}}$ and insert this form 
in~\eqref{equa}. We obtain 
\begin{equation} \label{equap_mut}
\left\{\begin{array}{l}
\displaystyle
\eps \D_t p_\eps(t,x,y) + \D_x\left[A(x,y)p_\eps(t,x,y)\right] 
+d(x,y)p_\eps(t,x,y) \phantom{\int}
\\
\displaystyle \hspace{4cm} =-(\rho_\eps(t)+\D_t u_\eps(t,y))p_\eps(t,x,y), 
\phantom{\int}
\\
\displaystyle 
A(x=0,y)\,p_\eps\left(t,x=0,y\right) \phantom{\int}
\\
\displaystyle \hspace{2cm}  =\frac{1}{\eps^n}\int_{\R^n}\int_ {\R_+}
{M(\frac{y'-y}{\eps})b(x',y')p_\eps(t,x',y')e^{\frac{u_\eps(t,y')-u_\eps(t,y)}{
\eps}}dx'dy'}, 
\\
\displaystyle \rho_\eps(t)=\int_{\R^n}\int_{\R_+}{m_\eps(t,x,y)dx dy},
\\
\displaystyle p_\eps(t=0,x,y)=p_\eps^0(x,y)>0. \phantom{\int}
\end{array}\right.
\end{equation}
With the change of variable $z=\frac{y'-y}{\eps}$, the renewal term is written 
as
\begin{multline}\label{boundarycondeps}
\displaystyle A(x=0)\,p_\eps(t,x=0,y)\\
=\int_{\R^n}\int_{\R_+}{M(z)e^{\frac{u_\eps(t,y+\eps 
z)-u_\eps(t,y)}{\eps}}b(x',y+\eps z)p_\eps(t,x',y+\eps z)dx'dz}.
\end{multline}
By taking formally the limit $\eps \to 0$, we obtain
\begin{equation*}
A(x=0)p(t,x=0,y)=\int_{\R^n}{M(z)e^{\nabla_y u(t,y) \cdot z}dz\int_{\R_+}b(x',y)p(t,x',y)dx'}.
\end{equation*}
Denoting
\begin{equation}
\eta(t,y):=\int_{\R^n}{M(z)e^{\nabla_y u(t,y) \cdot z}dz},
\end{equation}
the formal limit of \eqref{equap_mut} is written as
\begin{equation*}
\left\{\begin{array}{ll}
\D_x\left[A(x,y)p(t,x,y)\right] +d(x,y)p(t,x,y)=-(\rho(t)+\D_t u(t,y))p(t,x,y),
\\[1mm]
\displaystyle A(x=0)p(t,x=0,y)=\eta(t,y) \int_{\R_+}b(x',y)p(t,x',y)dx',
\\[1mm]
\displaystyle \rho(t)=\int_{\R^n}\int_{\R_+}{m(t,x,y)dx dy},
\\[1mm]
p(t=0,x,y)=p^0(x,y)>0, \quad u(t=0,y)=u^0(y).
\end{array}\right.
\end{equation*}
With this form, one can consider the following eigenproblem: for fixed $(y,\eta)\in\R^n\times(0,+\infty)$, find $(\Lambda(y,\eta),Q(x,y,\eta))$, solution of
\begin{equation}\label{valeurpropre}
\left\{\begin{array}{ll}
\D_x\left[A(x,y)Q(x,y,\eta)\right]+d(x,y)Q(x,y,\eta)=\Lambda(y,\eta)Q(x,y,\eta),\\[1mm]
\displaystyle A(x=0,y)Q(x=0,y,\eta)=\eta\int_{\R_+}{b(x',y)Q(x',y,\eta)dx'},\\[1mm]
\displaystyle Q(x,y,\eta)> 0, \quad \int_{\R_+} b(x,y) Q(x,y,\eta)dx=1.
\end{array}\right.
\end{equation}

 Using this eigenproblem, we will firstly compute the formal limit $u$ of the sequence $u_\eps$, and prove that it satisfies the following Hamilton-Jacobi equation 
\begin{equation}\label{eqfinale}
	\left\{\begin{aligned}
	&\D_t u(t,y)=-\Lambda\left(y,\int_{\R^n} M(z)e^{\nabla_y u(t,y)\cdot z}dz\right)-\rho(t),\quad t\geq 0, y\in\R^n,\\
	&u(0,y)=u^0(y), \   \    y\in\R^n.\\
	\end{aligned}\right.
\end{equation}
In this way, we formally recover the limit profile $p$ using \eqref{valeurpropre} with $\eta=\eta(t,y)$.
Back to the question of adaptive dynamics, $\Lambda(y,\eta(t,y))$ defines the effective fitness of the population with trait $y$.\\

In what follows, we study this limit problem and construct a solution $u$.
Actually the convergence of $p_\eps$ towards the solution $Q$ of the eigenproblem \eqref{valeurpropre} is an unsolved question. Indeed because of the particular form of the boundary condition \eqref{boundarycondeps}, we do not know how to study the asymptotic of $p_\eps$ as $\eps \to 0$. However, we construct a sequence $u_\eps$ from an approximation problem of \eqref{eqfinale} that is well defined and we prove it converges to the solution of \eqref{eqfinale} in the viscosity sense.

To begin with, we state the saturation of the population density, and the existence and uniqueness of the eigenelements of \eqref{valeurpropre}.



\subsection{Saturation and stationary problem}
As in the case without mutations in the previous section, it still holds that the total population is bounded.

\begin{prop}\label{saturation2}
We assume \eqref{hypothesed}--\eqref{hypM1bis} and \eqref{initialp}.
Then there exist two constants $\rho_m,\rho_M >0$ such that 
\begin{equation*}
\forall t\geq0,\   0<\rho_m\leq\rho_\eps(t)\leq\rho_M.
\end{equation*}
where $\rho_m:=\min (\underline{r}, \underline \rho^0)$ and $\rho_M:=\max (\bar r, \bar \rho^0)$. Hence, after extracting a subsequence, $\rho_\eps$ converges to a function $\rho$ in {weak*-$L^\infty(0,+\infty)$}. 
\end{prop}

We now establish the existence and uniqueness of the eigenelements in \eqref{valeurpropre}. Thus we introduce the associated dual problem: find $\Phi(x,y,\eta)$ solution of
\begin{equation}\label{EquationDualeBis}
\left\{\begin{aligned}
&A(x,y)\D_x\Phi(x,y,\eta)+\left[\Lambda(y,\eta)-d(x,y)\right]\Phi(x,y,\eta)=-\eta b(x,y)\Phi(0,y,\eta),\\
&\int_{\R^+}{Q(x,y,\eta)\Phi(x,y,\eta)dx}=1.
\end{aligned}\right.
\end{equation}
We also recall the definition \eqref{definitionF} for the function $F$.
The proof of the following theorem is given in Appendix~\ref{eigenappendix}.

\begin{thm}\label{theoremeeigen} 

We assume \eqref{hypothesed}--\eqref{bornesA}. Given $y\in\mathbb{R}^n$ and $\eta\in\mathbb{R}_+$, there exists a unique triplet $(\Lambda(y,\eta), Q(x,y,\eta), \Phi(x,y,\eta))$ solution of \eqref{valeurpropre} and \eqref{EquationDualeBis}. The map $x\mapsto Q(x,y,\eta)$ is bounded and integrable, $y \mapsto\Lambda(y,\eta)$ is $\mathcal{C}^1$ and we have
\begin{equation}\label{deriveelambda}
	\D_\lambda F>0,\quad F(y,\Lambda(y,\eta))=\frac{1}{\eta},
\end{equation}
\begin{equation}\label{detlambda}
\nabla_y\Lambda(y,\eta)=-\frac{\nabla_y F(y,\Lambda(y,\eta))}{\D_\lambda F(y,\Lambda(y,\eta))},\quad
\D_\eta\Lambda(y,\eta)=-\frac{1}{\eta^2\D_\lambda F(y,\Lambda(y,\eta))}<0.
\end{equation}
\end{thm}
In the sequel we consider the effective Hamiltonian (fitness)
\begin{equation}\label{DefinitionHamiltonian}
	H(y,p):=- \Lambda \big(y,\eta(p) \big),\quad\eta(p):=\int_{\R^n}{M(z)e^{p \cdot z}dz}>0.
	\end{equation}
Before constructing a solution to the associated Hamilton-Jacobi equation in the next section, we state the following result, which is proved in Appendix~\ref{HConvexity}.

\begin{prop}\label{PropConvexity}
The mapping $p\mapsto H(y,p)$ is convex, for all $y\in\R^n$.
\end{prop}

\subsection{The Hamilton-Jacobi equation}
\paragraph{} 
\label{par:}

Here we consider the Hamilton-Jacobi equation \eqref{eqfinale} that we may write from \eqref{DefinitionHamiltonian} as
\begin{equation*}
	\left\{
	\begin{aligned}
	&\D_t u(t,y)=H(y,\nabla_y u)-\rho(t),\\
	&u(0,y)=u^0(y), \   \    y\in\R^n.
	\end{aligned}\right.
\end{equation*}
Our goal is to build a solution to this equation. Therefore, we introduce $u_\eps$ solution of an approximate problem motivated by the form in \eqref{equap_mut}, which reads 
\begin{equation}\label{eqhamil}
\left\{\begin{array}{l}
\displaystyle \D_t u_\eps(t,y)=-\Lambda\left(y,\int_{\R^n}{M(z)e^{\frac{u_\eps(t,y+\eps z)-u_\eps(t,y)}{\eps}}d z}\right)-\rho_\eps(t),\\
u_\eps(0,y)=u_\eps^0(y), \   \    y\in\R^n.
\end{array}\right.
\end{equation} 
To simplify the Hamiltonian in equation \eqref{eqhamil}, we set
$
	U_\eps(t,y):=u_\eps(t,y)+\int_0^t\rho_\eps(t')d t,
$
which satisfies
\begin{equation}\label{HJUe}
	\D_t U_\eps(t,y)=-\Lambda\left(y,\int_{\R^n}{M(z)e^{\frac{U_\eps(t,y+\eps z)-U_\eps(t,y)}{\eps}}d z}\right).
\end{equation}
For clarity, we set 
\begin{equation*}
	\eta_\eps(t,y)=\int_{\R^n} M(z)e^{\frac{U_\eps(t,y+\eps z)-U_\eps(t,y)}{\eps}}dz.
\end{equation*}
We state the following theorem, which is the main result of this section. The set of assumptions $(\mathcal{H})$ is presented below.
\begin{thm}\label{stability}
Assuming $(\mathcal{H})$ there exists a unique solution $U_\eps$ to \eqref{HJUe}. Furthermore, $U_\eps$ converges locally uniformly to a function $U$ which is a viscosity solution of the equation
\begin{equation}\label{HJU}
	\D_t U(t,y)=H(y,\nabla_y U)=-\Lambda\left(y,\int_{\R^n}{M(z)e^{\nabla_y U \cdot z}d z}\right).
\end{equation}
\end{thm}
In other words, we prove a stability result in the language of the viscosity solutions theory (see~\cite{GB:94}) in a situation where the Hamiltonian depends on $\nabla_y U$ with an exponential growth, which is the main difficulty here. The plan of the proof is as follows. 
Firstly we consider the truncated equation associated to \eqref{HJUe}, for which classical results give existence and uniqueness of a global solution. Then we provide a uniform a priori estimate on the time derivative of the solution. It allows us to remove the truncation and to infer a global solution $U_\eps$ of \eqref{HJUe}. This proves the first step.

Secondly, we consider the semi-relaxed limits $\overline{U}:=\limsup U_\eps$ and $\underline{U}:=\liminf U_\eps$, and prove that they are respectively subsolution and supersolution of \eqref{HJU} in the viscosity sense. Then, an assumption of coercivity of $\eta\mapsto\Lambda(y,\eta)$ in \eqref{H2}, allows us to state that $\underline{U}$ is a Lipschitz function. Finally, using an uncommon uniqueness result on the Hamiltonian $H$, we prove that $\overline{U}=\underline{U}$, and conclude that $U_\eps$ converges locally uniformly to a viscosity solution of \eqref{HJU}.

\paragraph{Assumptions $(\mathcal{H})$.}
We assume \eqref{initialu}. In addition, for any compact interval $I$, we assume there exist two constants $L_0$, $L_1>0$, (depending on $I$) such that
\begin{equation}\label{H1}
\forall y\in\R^n, \forall \eta\in I,\quad
\left\{\begin{aligned}
&\vert\Lambda(y,\eta)\vert\leq L_0,\\
&\vert\D_\eta\Lambda(y,\eta)\vert\leq L_1.\\
\end{aligned}\right.
\end{equation}
We also assume
	\begin{equation}\label{H2}
		\vert\Lambda(y,\eta)\vert\to+\infty \text{ when }\eta\to+\infty\text{ or }\eta\to0, \text{ uniformly in } y\in\R^n.
	\end{equation}
Finally, the following assumption is required for our uniqueness result, stated in Theorem~\ref{uniqueness}. For all compact set $K_p \subset \R^n$, we assume there exist $C>0,\gamma_1\in[0,4),\gamma_2\in[0,1)$ such that
	\begin{equation}
		\begin{aligned}\label{LemmeH}
			&\forall y\in\R^n, \forall p \in K_p,
			& \left\{\begin{aligned}
				\vert \nabla_y H(y,p)\vert\leq C\left(1+\vert y\vert^{\gamma_1}\right),\\
				\vert \nabla_p H(y,p)\vert\leq C\left(1+\vert y\vert^{\gamma_2}\right).
			\end{aligned}\right.
		\end{aligned}
	\end{equation}

\subsection{Global existence and a priori estimate} 

This section is devoted to the proof of the following Theorem, which is the first step towards Theorem~\ref{stability}.
\begin{thm}\label{ThmEpsilon}
Assume \eqref{H1}.
Then, for all $\eps>0$, there exists a unique global solution $U_\eps$ to the equation~\eqref{HJUe}, such that $\vert \D_tU_\eps(t,y)\vert\leq L$ for a constant $L>0$, uniformly in $\eps>0,\ t>0,\ y\in\R^n$.
\end{thm}
\subsubsection{The truncated problem} 

We first consider a truncated problem associated to \eqref{HJUe}. For a fixed $R>0$, we define the function $\phi_R:\R\to\R$ which is smooth, increasing and satisfies the following conditions:
\begin{itemize}
\item $\phi_R(r)= r \text{ for }r \in[-\frac{R}{2},\frac{R}{2}]$,
\item $\phi_R(r)=R \text{ for }r \geq 2R$,
\item $\phi_R(r)= -R \text{ for } r \leq -2R$,
\item $\phi_R'\geq 0$ is uniformly bounded.
\end{itemize}	
Let $\eps>0$ be fixed. We consider the Cauchy problem
\begin{equation}\label{eqR}
	\left\{
	\begin{aligned}
		&\D_t U^R_\eps(t,y)=\phi_R\left(-\Lambda\left(y,\int_{\R^n}{M(z)e^{\frac{U^R_\eps(t,y+\eps z)-U^R_\eps(t,y)}{\eps}}d z}\right)\right),\\
		&U^R_\eps(0,\cdot)=u^0_\eps.  	
	\end{aligned}\right.
\end{equation}
We state the following result
\begin{Lemme}\label{cauchyR}
Assuming \eqref{H1}, there exists a unique solution of \eqref{eqR}, defined globally in time.
\end{Lemme}
\noindent The proof is based on the Cauchy-Lipschitz Theorem and uses only classical arguments. It is left to the reader.

\subsubsection{Estimate on the time derivative}
The particular form of \eqref{eqR} allows us to infer uniform a priori estimates on $\D_tU_\eps^R$. It is stated in the following result.
\begin{prop}\label{borneV}
For all $R>0,\eps>0$, we have $$\Vert\D_t U_\eps^R\Vert_\infty\leq\Vert\D_tu^0_\eps\Vert_\infty:=\left\Vert\Lambda(y,\eta_\eps(0,y))\right\Vert_\infty.$$
As a consequence, there exists a positive constant $L$, independent of $R$ and $\eps$ such that
\begin{equation}
\forall \eps >0,\forall R>0, \forall t \geq 0, \forall y \in \R^n, \quad \vert \D_t U_\eps^R (t,y) \vert \leq L.
\end{equation}
\end{prop}
\noindent The complete proof is postponed to Appendix~\ref{AppAPrioriEstimate}. However we give the formal idea here. As $R$ is fixed, we simply write $U_\eps$ instead of $U_\eps^R$.
We set $V_\eps(t,y):=\D_t U_\eps(t,y).$
Differentiating \eqref{HJU} with respect to $t$, we obtain
\begin{equation}\label{formal_time}
\D_tV_\eps(t,y)=\int_{\R^n} K_\eps (t,y,z) \left(\frac{V_\eps(t,y+\eps z)-V_\eps(t,y)}{\eps}\right) dz,
\end{equation}
where $K_\eps(t,y,z):=-\D_\eta\Lambda\left(y,\eta_\eps(t,y)\right) 
M(z)e^{\frac{U_\eps(t,y+\eps z)-U_\eps(t,y)}{\eps}} $. Note that, thanks to 
\eqref{detlambda}, $K_\eps$ is positive. Then, if for some $t>0$, $V_\eps 
(t, \cdot)$ reaches its maximum at $\bar y\in\R^n$, we obtain the inequality 
\begin{equation*}
\D_t V_\eps(t, \bar y) =\int_{\R^n} K_\eps (t,\bar y,z) \left(\frac{V_\eps(t,\bar y+\eps z)-V_\eps(t,\bar y)}{\eps}\right) dz \leq 0.
\end{equation*}
Formally, it shows that the maximum value of $V_\eps$ is decreasing with time, that is, 
\[
\sup_y V_\eps(t, y) \leq \sup_y V_\eps(0, y)= \sup_{y}\D_tu^{0}_\eps.
\]
With the same method we show
${\inf_y\D_tU_\eps\geq\inf_y\D_tu^{0}_\eps}$, which completes the first step of the proof.
Then, using \eqref{H1} and that $u^0_\eps$ is a Lipschitz function from \eqref{initialu}, we deduce an estimate on $\D_tU_\eps$, uniform in $R>0$ and $\eps>0$.

\subsubsection{Removing the truncation} 
\label{ssub:}

From Proposition~\ref{borneV},  $\D_t U_\eps^R(t,y)=\phi_R\left(-\Lambda\left(y,\eta_\eps(t,y)\right)\right)$ is bounded uniformly in $R$. As $\phi_{R}\equiv \mathrm{Id}$ on $[-\frac{R}{2},\frac{R}{2}]$, then, for $R$ large enough, $U_\eps^R$ is also solution to the non-truncated problem~\eqref{HJUe}. Conversely, a solution to~\eqref{HJUe} with a bounded time derivative is a solution to the truncated problem~\eqref{eqR} for $R$ large enough. Thus $U_\eps:=U_\eps^R$ is the unique solution of~\eqref{HJUe} with $\Vert \D_t U_\eps\Vert_\infty\leq L$, for $R$ large enough. The proof of Theorem~\ref{ThmEpsilon} is thereby complete.

\subsection{The semi-relaxed limits} 

We assume \eqref{H1}. Thanks to Theorem~\ref{ThmEpsilon}, there exists a constant $C>0$ such that
\begin{equation}\label{croissancet}
\begin{aligned}
		&\vert U_\eps(t,y)\vert \leq \vert u_\eps^0(y)\vert+ Lt \leq C+L t+k_0\vert y \vert, &\forall t>0,\ \forall y\in\R^n,
	\end{aligned}                     
\end{equation}
uniformly in $\eps>0$. This allows us to consider the following semi-relaxed limits (see \cite{GB.BP:88, HI})

\begin{equation}\label{defsupersub}
	\overline{U} (t,y)=\limsup\limits_{
	\substack{
		x\to y\\
		s\to t\\
		\eps\to0}}
	U_\eps(s,x),  \quad
	\underline{U}(t,y)=\liminf\limits_{
	\substack{
		x\to y\\
		s \to t\\
		\eps\to0}}
	U_\eps(s,x).
\end{equation}
Note that accordingly $\underline{U}$ and $\overline{U}$ satisfy the inequality \eqref{croissancet}. More precisely, from the uniform estimate on the time derivative stated in Theorem~\ref{ThmEpsilon} we have
\begin{equation}\label{LipschitzT}
\vert \overline{U}(t,y)-u^0(y)\vert\leq Lt,\quad \vert \underline{U}(t,y)-u^0(y)\vert\leq Lt.
\end{equation}

In this section, we prove
\begin{thm}\label{uniqueness}
Assuming \eqref{H1}--\eqref{LemmeH}, we have $\overline{U}=\underline{U}$. 
\end{thm}
\noindent This result implies that $U_\eps$ converges locally uniformly to a solution $U$ of equation \eqref{HJU}, which completes the proof of Theorem~\ref{stability}.

\subsubsection{Subsolution and supersolution}
The following proposition is adapted from classical stability results for viscosity solutions of Hamilton-Jacobi equations (see \cite{GB:94}). Note that it slightly differs from the usual framework because of the nonlocal term $\eta_\eps(t,y)$.
\begin{prop}\label{subsuper}
	The semi-continuous functions $\overline{U}$ and $\underline{U}$ defined in \eqref{defsupersub} are respectively subsolution and supersolution of~\eqref{HJU} in the viscosity sense in $(0,\infty)\times\R^n$. Also, for all $T>0$, the viscosity inequalities stand for ${t\in (0,T]}$.
\end{prop}

\begin{proof}[Proof of Proposition \ref{subsuper}.]
 In order to prove that $\overline{U}$ is a viscosity subsolution of~\eqref{HJU}, since $\overline{U}$ is upper semi-continuous, let us consider a test function $\varphi$ and a point $(t_0,y_0)$ such that $
	\overline{U}-\varphi$ reaches a global maximum at $(t_0,y_0)$. From classical results, there exists $(t_\eps,y_\eps)$ such that 
	\begin{equation*}
		\left\{\begin{aligned}
			&(t_\eps,y_\eps)\underset{\eps \to 0}{\longrightarrow}(t_0,y_0),\\
			&\max_{t,y} U_\eps-\varphi =(U_\eps-\varphi)(t_\eps,y_\eps).
		\end{aligned}\right.
	\end{equation*}
Besides, note that for all $z\in\R^n$, ${\varphi(t_\eps,y_\eps+\eps z)-U_\eps(t_\eps,y_\eps+\eps z)\geq \varphi(t_\eps,y_\eps)-U_\eps(t_\eps,y_\eps)}$, thus we have
				\begin{equation*}
					\frac{\varphi(t_\eps,y_\eps+\eps z)-\varphi(t_\eps,y_\eps)}{\eps}\geq\frac{U_\eps(t_\eps,y_\eps+\eps z)- U_\eps(t_\eps,y_\eps)}{\eps}.
				\end{equation*}
Since $\D_\eta \Lambda<0$ from \eqref{detlambda}, equation \eqref{HJUe} gives
\begin{multline*}
\D_t\varphi(t_\eps,y_\eps)=-\Lambda\left(y_\eps,\int_{\R^n} 
M(z)e^{\frac{U_\eps(t_\eps,y_\eps+\eps z)-U_\eps(t_\eps,y_\eps)}{\eps}}d 
z\right)\\
\leq -\Lambda\left(y_\eps,\int_{\R^n} 
M(z)e^{\frac{\varphi(t_\eps,y_\eps+\eps z)-\varphi(t_\eps,y_\eps)}{\eps}}d 
z\right).
\end{multline*}
As $\eps$ goes to $0$,
	\begin{equation*}
		\D_t \varphi(t_0,y_0)\leq-\Lambda\left(y_0,\int_{\R^n} M(z)e^{\nabla_y \varphi(t_0,y_0)\cdot z}\right)=H(y_0,\nabla_y \varphi(t_0,y_0)),
	\end{equation*}
then $\overline{U}$ is a viscosity subsolution of \eqref{HJU}. 
With the same method, we prove that ${\underline {U}}$ is a viscosity 
supersolution. It completes the first part of the proof. The second part of the 
statement is a well-known result, and a proof can be found in \cite{GB:94}.
\end{proof}

\subsubsection{A posteriori Lipschitz estimate on $\underline{U}$}
The announced Lipschitz continuity of $\underline{U}$ is stated in the following result.

\begin{prop}\label{apriorigrad}
Assume \eqref{H1}--\eqref{H2}. Then the lower semi-continuous function $\underline{U}$ defined in~\eqref{defsupersub} is a $L$-Lipschitz function with  $L>0$ defined below. 
\end{prop}
\noindent We first prove these two preliminary lemmas. We point out that \eqref{H2} plays a crucial role in the proof.
\begin{Lemme}\label{LemmePreliminaire1}
Assume \eqref{H1}--\eqref{H2}. Then there exist some positive constants $\underline{\eta},\bar \eta, L_1$ such that, uniformly in $\eps$, $\forall (t,y)\in (0,+\infty)\times\R^n,$
\begin{gather}
\underline{\eta}\leq \eta_\eps(t,y)\leq\bar \eta,\label{BornesEta}\\
\left\vert\D_\eta\Lambda\left(y,\eta_\eps(t,y)\right)\right\vert\leq L_1.\label{LambdaEta}
\end{gather}
\end{Lemme}
\begin{proof}
From Theorem~\ref{ThmEpsilon}, we know $\D_tU_\eps(t,y) = - \Lambda(y, \eta_\eps(t,y))$ is bounded for $(t,y)\in(0,+\infty)\times\R^n$, uniformly in $\eps>0$. From~\eqref{H2}, we deduce that $\eta_\eps(t,y)$ is bounded, which proves~\eqref{BornesEta}. Then we derive~\eqref{LambdaEta} directly from assumption~\eqref{H1}.
\end{proof}

\noindent In what follows, we use the notation $\nabla U=(\D_t U,\nabla_y U)$. 

\begin{Lemme}\label{lemmegrad} In the viscosity sense, $\nabla \underline{U}$ is bounded, that is, there exists a constant $L\geq k_0$ such that if $\psi$ is a smooth function and ${\underline{U}-\psi}$ reaches its minimum at $(t_0,y_0)\in(0,+\infty)\times\R^n$, then
\begin{equation*}
	\vert \D_t \psi(t_0,y_0)\vert \leq L, 
\end{equation*}
\begin{equation*}
	\|\nabla_y \psi(t_0,y_0)\|_{\infty} \leq L. 
\end{equation*}
\end{Lemme}

\begin{proof} 
Let $\psi$ be a smooth function such that ${\underline{U}-\psi}$ reaches its 
minimum at $(t_0,y_0).$
Similarly to the proof of Proposition~\ref{subsuper}, up to extraction of a 
subsequence, there exists a sequence of minimum points $(t_\eps,y_\eps)$ of 
$U_\eps - \psi$ which converges to $(t_0,y_0)$. As $\underline{U}$ is a 
supersolution, we obtain
\begin{multline}\label{ineq_lip}
-\Lambda\left(y_\eps,\int_{\R^n} M(z)e^{\frac{\psi(t_\eps,y_\eps+\eps 
z)-\psi(t_\eps,y_\eps)}{\eps}}d z\right)\\ 
\leq \D_t\psi(t_\eps,y_\eps)=\D_tU_\eps(t_\eps,y_\eps)=-\Lambda\left(y_\eps,
\eta_\eps(t_\eps,y_\eps)\right) .		
\end{multline}
From the estimate on $\D_t U_\eps$ given by Theorem~\ref{ThmEpsilon}, we have, when $\eps$ goes to $0$,
\begin{equation}
\vert\D_t\psi(t_0,y_0) \vert\leq L.
\end{equation}
Thus, from $\D_\eta\Lambda<0$, \eqref{BornesEta} and \eqref{ineq_lip}, we derive, as $\eps$ goes to 0, 
\begin{equation*}	\int_{\R^n} M(z)e^{\nabla_y\psi(t_0,y_0)\cdot z}d z\leq \overline{\eta}.
\end{equation*}
Since $M(z)>0$, we deduce
\begin{equation}
\|\nabla_y\psi(t_0,y_0)\|_{\infty}\leq L',
\end{equation}
for some constant $L'$. Setting $L:=\max(L,L',k_0)$ achieves the proof.
\end{proof}

\begin{proof}[Proof of Proposition \ref{apriorigrad}.]We want to prove that
	\begin{equation*}
		\begin{aligned}
				&\forall (t,t')\in (0,\infty)^2,(y,y')\in(\R^n)^2, & \underline{U}(t',y')-\underline{U}(t,y)\leq L(\vert t-t' \vert +\vert y-y'\vert).
		\end{aligned}
	\end{equation*}
By contradiction, we assume that there exists $K>L$ such that, for some 
$(t_0,t_0')\in (0,\infty)^2$ and $(y_0,y_0')\in(\R^n)^2$,
\begin{equation}\label{absurd}
\underline{U}(t_0',y_0')-\underline{U}(t_0,y_0)- 
K(\vert t_0-t_0' \vert +\vert y_0-y_0'\vert)>0.
\end{equation}
Let us define the test function
$
		\psi (t,y) := \underline{U}(t_0',y_0')-K(\vert t-t_0' \vert +\vert y-y_0'\vert).
$
	As $k_0<K$, from~\eqref{croissancet} we derive that $\psi(t,y)-\underline{U}(t,y)\to-\infty$ when $\vert y\vert\to\infty$. Because this function is upper semicontinuous, it reaches its maximum at a point $(\bar t,\bar y)\in[0,\infty)\times\R^n$. In order to apply Lemma~\ref{lemmegrad} at $(\bar t,\bar x)$, we have to prove that $\bar t>0$ and that $\psi$ is smooth in a neighborhood of $\bar x$. 
We prove the first assertion by contradiction. We assume $\bar t=0$. From~\eqref{LipschitzT} and the Lipschitz continuity of $u^0$, we have
\begin{equation}
\underline{U}(t_0',y_0')-\underline{U}(\bar t, \bar y)
\leq \underline{U}(t_0',y_0')-u^0(y_0')+u^0(y_0')-u^0(\bar y)
\leq L (\vert t_0'-\bar t \vert+\vert \bar y - y_0'\vert),
\end{equation}
which contradicts \eqref{absurd}. Thus $\bar t>0$. Besides, using~\eqref{absurd} we deduce $\bar x\neq x_0$, therefore $\psi$ is smooth in a neighborhood of $\bar x$. Thus we can apply Lemma~\ref{lemmegrad} and obtain $\| \nabla \psi(\bar t, \bar y) \|_{\infty}=K\leq L,$ which is a contradiction.
\end{proof}

\subsection{Uniqueness result}

We prove Theorem \ref{uniqueness}. This implies that $U_\eps$ converges locally uniformly to a function $U$ solution of~\eqref{HJU} in the viscosity sense. Therefore, it completes the proof of Theorem~\ref{stability}.

 In fact, we prove that a Lipschitz continuous supersolution remains above a subsolution provided it is the case at initial time. Namely, we prove $\underline U\equiv\overline U$, with the notations introduced in~\eqref{defsupersub}. We point out that this uniqueness result is not standard since our assumption~\eqref{LemmeH} allows the Hamiltonian to have superlinear growth. The fact that $\underline{U}$ is Lipschitz continuous, as stated in Proposition~\ref{apriorigrad}, is used as a key ingredient.

\begin{proof}[Proof of Theorem \ref{uniqueness}.] From the definition of $\overline U$ and $\underline U$ given in \eqref{defsupersub}, we know that $\underline U\leq\overline U$. We prove the reverse inequality. We fix $T>0$. By contradiction, we assume  
\begin{equation}\label{contradiction}
	\sigma:=\sup_{\substack{y\in\R^n\\t\in[0,T]}}(\overline{U}(t,y)-\underline{U}(t,y))>0.
\end{equation}
From \eqref{croissancet} and \eqref{defsupersub}, there exists a constant $C>0$ such that
\begin{equation}\label{bornety}
	\begin{aligned}
		&\forall t>0, \forall y\in\R^n, &\vert \underline{U}(t,y)\vert + \vert \overline{U}(t,y)\vert \leq C+2k_0\vert y\vert,
	\end{aligned}
\end{equation}
The same estimate also holds for $\underline{U}$.
We use the classical method of doubling the variables in the framework of viscosity solutions (see \cite{Cr-Ev-Li,crandall1992}). Let us fix $\alpha>0$, $\delta\in[0,1]$  and set for all $t\in[0,T], t'\in[0,T], y\in\R^n, y'\in\R^n,$ we define

\begin{equation}
	V_\delta(t,y,t',y'):=\left[\overline{U}(t,y)-\alpha t-\delta\vert y\vert^2\right]-\left[\underline{U}(t',y')+\alpha t'+\delta\vert y'\vert^2\right]-\frac{\vert y-y'\vert^2}{\delta^2}-\frac{\vert t-t'\vert^2}{\delta^2}.
\end{equation}
Thanks to \eqref{bornety}, $V_\delta$ reaches its maximum $M_\delta$ at a point ${(t_\delta, y_\delta, t'_\delta, y'_\delta)}$. In what follows we use the following lemma.

\begin{Lemme}\label{LemmeDivers}
When $\delta$ vanishes, the estimates hold \begin{enumerate}
		\item $\vert t_\delta -t_\delta'\vert, \vert y_\delta -y_\delta'\vert=O(\delta^2)$,
		\item $ \vert y_\delta\vert, \vert y_\delta'\vert = O(\frac{1}{\sqrt{\delta}}),$
		\item $\displaystyle \liminf_{\delta \to 0}t_\delta, \liminf_{\delta \to 0}t_\delta'>0$.
	\end{enumerate}
\end{Lemme}
The proof of Lemma \ref{LemmeDivers} is essentially technical. Note that the Lipschitz continuity of $\underline{U}$ is a key ingredient, since usual estimates cannot give any better result than $\vert y_\delta-y_\delta'\vert=O(\delta)$.

\begin{proof}[Proof of Lemma \ref{LemmeDivers}.]
First, we prove that $\vert y_\delta\vert, \vert y_\delta'\vert=O(\frac{1}{\delta}).$ 
For simplicity, all constants that do not depend on $\delta$ are denoted by $K$. 
We have
\begin{equation*}
	\begin{aligned}
			\forall(t,y,t',y')\in\left([0,T]\times\R^n\right)^2,		V_\delta(t,y,t',y')&\leq K+k_0(|y|+|y'|)-\delta (|y|^2+|y'|^2)\\
			&\leq K+Kz-\delta z^2,
	\end{aligned}
\end{equation*}
where $z=\max(\vert y\vert,\vert y'\vert)$. This means that $V_\delta$ can be bounded from above by a second order polynomial function of $z$. Consequently, the points $(y_\delta,y_\delta')$ where $V_\delta$ reaches its maximum are bounded by $z_0$, maximum solution to the equation
\begin{equation*}
	V_\delta(0,0,0,0)-1= K+Kz-\delta z^2,
\end{equation*}
which writes under the form
\begin{equation*}
	z_0 = \frac{K+\sqrt{K+\delta K}}{\delta}=O(\frac{1}{\delta}).
\end{equation*}
Thus we infer 
\begin{equation}\label{boundy}
	\vert y_\delta\vert,\   \vert y_\delta'\vert =O(\frac{1}{\delta}).
\end{equation}

Now we prove the assertion 1 of Lemma \ref{LemmeDivers}.
As $M_\delta\geq V_\delta(t_\delta,y_\delta,t_\delta,y_\delta)$, we have
\begin{multline}\label{reference2}
\alpha (t_\delta'-t_\delta)+\delta(\vert y_\delta'\vert^2-\vert 
y_\delta\vert^2)+\frac{\vert y_\delta-y_\delta'\vert^2}{\delta^2}+\frac{\vert 
t_\delta-t_\delta'\vert^2}{\delta^2} \\
\leq \overline 
U(t_\delta,y_\delta)-\underline{U}(t_\delta',y_\delta')
-\overline U(t_\delta,y_\delta)+\underline{U}(t_\delta,y_\delta) \\
\leq \underline U(t_\delta,y_\delta)-\underline U(t_\delta',y_\delta')
\leq L\left(\vert t_\delta-t_\delta'\vert+\vert 
y_\delta-y_\delta'\vert\right),
\end{multline}
from the Lipschitz continuity of $\underline{U}$ stated in Proposition~\ref{apriorigrad}.
Besides, from \eqref{boundy} we obtain
	\begin{equation}\label{ineq_maxpoint}
	 \delta\left(\vert y_\delta\vert^2-\vert y_\delta'\vert^2\right)\leq \delta (|y_\delta| +|y_\delta'|)(|y_\delta| -|y_\delta'|)\leq K\vert y_\delta-y_\delta'\vert.
	\end{equation}
	Consequently, using \eqref{ineq_maxpoint} in \eqref{reference2}, we have
	\begin{equation}\label{Bootstrap}
		\frac{\vert y_\delta-y_\delta'\vert^2}{\delta^2}+\frac{\vert t_\delta-t_\delta'\vert^2}{\delta^2}\leq K\left(\vert t_\delta-t_\delta'\vert+\vert y_\delta-y_\delta'\vert\right).
	\end{equation}
Then, using the inequality
\begin{equation*}
(\vert y_\delta-y_\delta'\vert + \vert t_\delta-t_\delta'\vert)^2 \leq 2 (\vert y_\delta-y_\delta'\vert^2 + \vert t_\delta-t_\delta'\vert^2),
\end{equation*}
we deduce
\begin{equation*}
\frac{(\vert y_\delta-y_\delta'\vert+\vert t_\delta-t_\delta'\vert)^2}{\delta^2}\leq K\left(\vert t_\delta-t_\delta'\vert+\vert y_\delta-y_\delta'\vert\right).
\end{equation*}
Then, we obtain the estimates
\begin{equation}\label{estimates_delta}
\vert y_\delta-y_\delta'\vert,\vert t_\delta-t_\delta'\vert=O\left(\delta^{2}\right).
\end{equation}
Hence the assertion 1 of Lemma \ref{LemmeDivers}.
\\

Next, we prove the assertion 2. From $M_\delta\geq V_\delta(0,0,0,0)$, \eqref{LipschitzT} and Proposition~\ref{apriorigrad} we infer
\begin{equation}\label{DerniereEstimee}
\begin{aligned}
\alpha (t_\delta+t_\delta')+\delta\left(\vert y_\delta\vert^2+\vert y_\delta'\vert^2\right)
&\leq& &\overline{U}(t_\delta,y_\delta)-\underline{U}(t_\delta',y_\delta')\\
		 &=& &\left[\overline U(t_\delta,y_\delta)-u^0(y_\delta)\right]+\left[u^0(y_\delta)-\underline U(t_\delta,y_\delta)\right] \\
		 && &+\left[ \underline U(t_\delta,y_\delta)-\underline{U}(t_\delta',y_\delta')\right]\\
		 &\leq& &2Lt_\delta +L(\vert t_\delta-t_\delta'\vert+\vert y_\delta-y_\delta'\vert).
\end{aligned}
\end{equation}
We deduce $\delta(\vert y_\delta\vert^2+\vert y_\delta'\vert^2)=O(1)$, hence the assertion 2.
\\

Finally, we prove the last assertion. By contradiction, we assume, up to extraction of a subsequence, that $t_\delta'\to0$ as $\delta$ goes to $0$. From \eqref{estimates_delta}, we deduce that $t_\delta$ converges to 0  as $\delta$ vanishes and then, from \eqref{DerniereEstimee}, we obtain $\delta(\vert y_\delta\vert^2+\vert y_\delta'\vert^2)=o(1)$. We set $M:=\max_{\substack{(t,y)\in[0,T]\times\R^n}}V_\delta(t,y,t,y)$ and choose $\delta$ and $\alpha$ small enough to ensure $M\geq \frac{\sigma}{2}$. We write
	\begin{equation*}
		\begin{aligned}
			\frac{\sigma}{2}\leq M\leq M_\delta
			 &\leq\overline U (t_\delta,y_\delta)-\underline{U}(t_\delta',y_\delta')\\
			 &=[\overline U (t_\delta,y_\delta)-u^0(y_\delta)]+[u^0 (y_\delta)-u^0(y_\delta')]+[u^0 (y_\delta')-\underline{U}(t_\delta',y_\delta')]\\
			 &\leq L(t_\delta+t_\delta')+k_0\vert y_\delta-y_\delta '\vert,
		\end{aligned}
	\end{equation*}
where we used \eqref{LipschitzT} for the last inequality. As $\delta$ goes to $0$, we deduce from the previous inequality that $\sigma\leq0$, contradiction. Thus $t_\delta'>0$ uniformly in $\delta$ when $\delta$ goes to $0$. Moreover we have $t_\delta-t_\delta'=o(1)$, hence the result.
\end{proof}

\noindent Now we go back to the proof of Theorem \ref{uniqueness}. We use that $\overline{U}$ and $\underline{U}$ are subsolution and supersolution in the viscosity sense. We define the test function
\begin{equation*}
	\varphi_{\alpha,\delta}(t,y):=\alpha t+\delta \vert y\vert^2+\left[\underline{U}( t_\delta', y_\delta')+\alpha  t_\delta'+\delta\vert y'_\delta\vert^2\right]+\frac{\vert y- y_\delta'\vert^2}{\delta^2}+\frac{\vert t- t_\delta'\vert^2}{\delta^2},
\end{equation*}
which is smooth and is such that $\overline{U}-\varphi_{\alpha,\delta}$ reaches its global maximum at the point $(t_\delta, y_\delta)$. Since $\overline{U}$ is a subsolution of~\eqref{HJU} ands $t_\delta\in(0,T]$,  the viscosity inequality holds
\begin{equation*}
	\D_t\varphi_{\alpha, \delta}(t_\delta, y_\delta)=\alpha+\frac{2}{\delta^2}(  t_\delta- t_\delta')\leq H\left( y_\delta, 2\delta y_\delta+\frac{2}{\delta^2}(  y_\delta- y_\delta')\right).
\end{equation*}
In the same way, since $\underline{U}$ is a supersolution, we derive
\begin{equation*}
	-\alpha+\frac{2}{\delta^2}  (t_\delta- t_\delta')\geq H\left( y_\delta', -2\delta y_\delta'+\frac{2}{\delta^2}( y_\delta- y_\delta')\right).
\end{equation*}
Subtracting this last inequality from the previous one and using Lemma~\ref{LemmeDivers}, we obtain
\begin{equation*}
	\begin{aligned}
			2\alpha&\leq H\left( y_\delta, 2\delta y_\delta+\frac{2}{\delta^2}  (y_\delta- y_\delta')\right)-H\left( y_\delta', -2\delta y_\delta'+\frac{2}{\delta^2}(  y_\delta- y_\delta')\right)\\
			&\leq \left[H\left( y_\delta, 2\delta y_\delta+\frac{2}{\delta^2}(  y_\delta- y_\delta')\right)-H\left( y_\delta, -2\delta y_\delta'+\frac{2}{\delta^2}(  y_\delta- y_\delta')\right)\right]\\
				 &\quad+\left[H\left( y_\delta, -2\delta y_\delta'+\frac{2}{\delta^2}(  y_\delta- y_\delta')\right)-H\left( y_\delta', -2\delta y_\delta'+\frac{2}{\delta^2}(  y_\delta- y_\delta')\right)\right]\\
				 &\leq 2\delta C(1+\vert y_\delta\vert^{\gamma_2})\vert y_\delta+ y_\delta'\vert
				 +C(1+\vert y_\delta\vert^{\gamma_1}+\vert y_\delta'\vert^{\gamma_1})\vert y_\delta-y_\delta'\vert\\
				 &=O(\delta^{\frac{1}{2}-\frac{\gamma_2}{2}})+O(\delta^{2-\frac{\gamma_1}{2}}),
		\end{aligned}
\end{equation*}
From assumption \eqref{LemmeH} we have ${2\alpha=o(1)}$ as $\delta$ goes to $0$, thus we find $\alpha\leq0$, which is a contradiction. Therefore $\sigma=0$ and we have
$
	\overline{U}= \underline{U}.
$ The proof of Theorem~\ref{uniqueness} is thereby complete.
\end{proof}

\section{Numerical simulations}
\label{numerics}

In order to complete the theory, we present numerical results in the case 
without mutations studied in Section~\ref{secCaseWM}. We perform a simulation 
of equation \eqref{systemesansmutation} with $\eps=5 \cdot 10^{-3}$. The 
numerical results allow to visualize $u_\eps$ and then the concentration 
dynamics of the population density. We choose the variable pair $(x,y)$ to be in 
the set $[0,1] \times [0,4]$ which we discretize with the steps $\Delta 
x=\frac{1}{M}$ and $\Delta y=\frac{1}{N}$ with $M=90, N=40$. The time step 
$\Delta t$ is chosen to be $5 \cdot 10^{-5}$ according to the CFL condition. We 
denote by $m^k_{i,j}$ the numerical solution at grid point $x_i=i \Delta x$, 
$y_j=j \Delta y$ and time $t_k=k \Delta t$. Equation \eqref{systemesansmutation} 
is solved by an implicit-explicit finite-difference method with the following 
scheme: for $i=1, \dots, N$ and $j=1, \dots, M$,
\begin{equation}\label{scheme}
m^{k+1}_{i,j}=m^k_{i,j} -\frac{\Delta t}{\eps} \frac{A(x_i,y_j)m^k_{i,j}-A(x_{i-1},y_j)m^k_{i-1,j}}{\Delta x}-\frac{\Delta t}{\eps}\left( \rho^k m^k_{i,j} -d(x_i,y_j) m^{k+1}_{i,j} \right),
\end{equation}
and the boundary term is discretized as
\begin{equation*}
A(0,y_j)m^{k+1}_{0,j}= \sum_{i=1}^M b(x_i,y_j) m^k_{i,j},
\end{equation*}
which is necessary for computing when $i=0$ in \eqref{scheme}.

The numerics is performed using Matlab with parameters as follows. We choose the initial number of individuals to be $1000$ and the final time $T=1.5$. We choose the following functions $A, b$ and $d$ as follows

\begin{equation*}
A(x,y)=1, \qquad b(x,y)=10 \cdot \frac{y}{1+x^2}, \qquad d(x,y)= y^3 \cdot (2+x/3),
\end{equation*}
and the initial data
\begin{equation*}
m^0(x,y)=p^0(x,y) e^{\frac{u^0(y)}{\eps}}, 
\end{equation*}
with
\begin{equation*}
p^0(x,y)= \exp{(-0.8 x)}, \qquad u^0(y)= -\frac{(y-0.5)^2}{2}.
\end{equation*}

\begin{figure}[!h]
\centering
\includegraphics[width=0.45\textwidth]{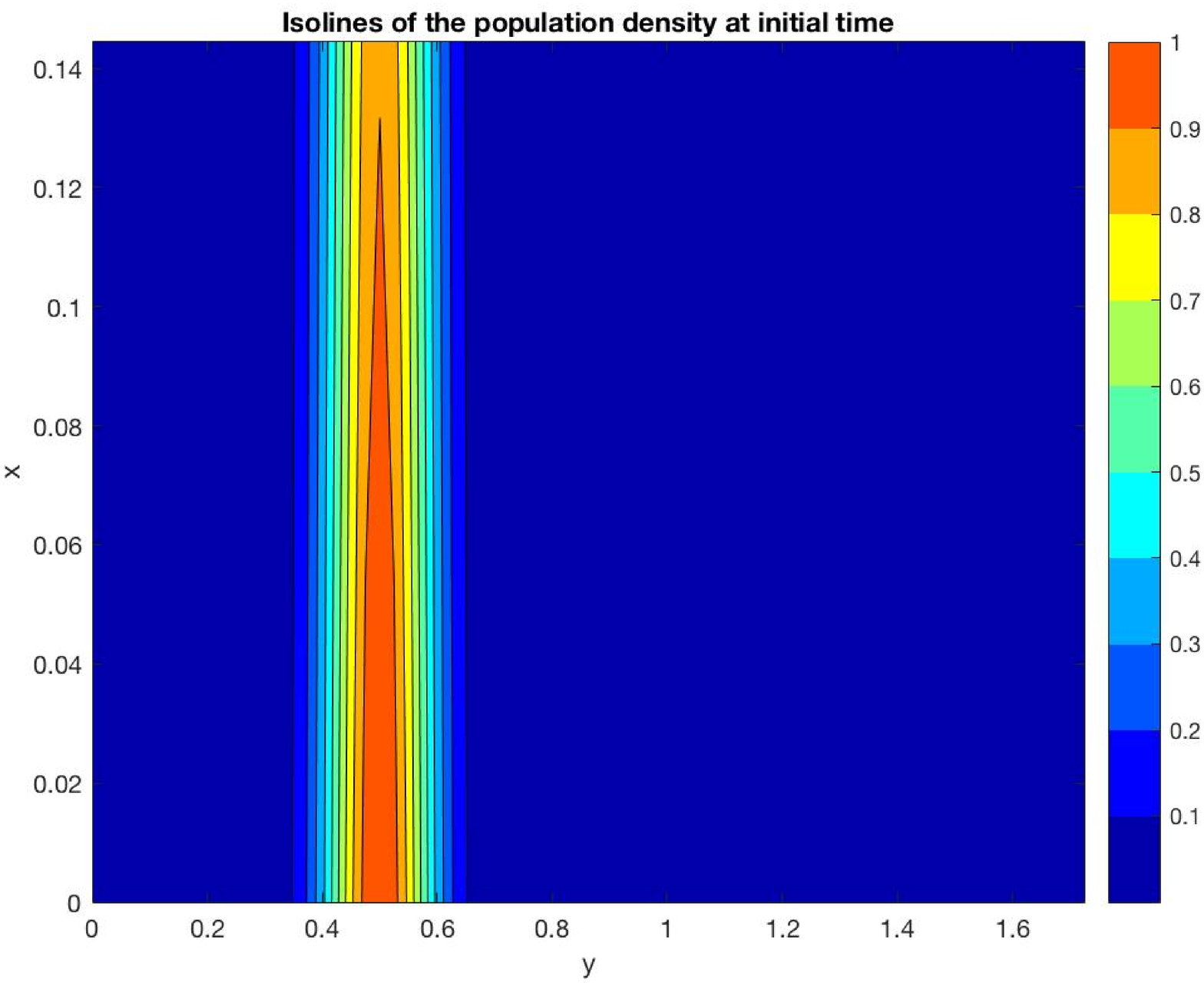}
\includegraphics[width=0.45\textwidth]{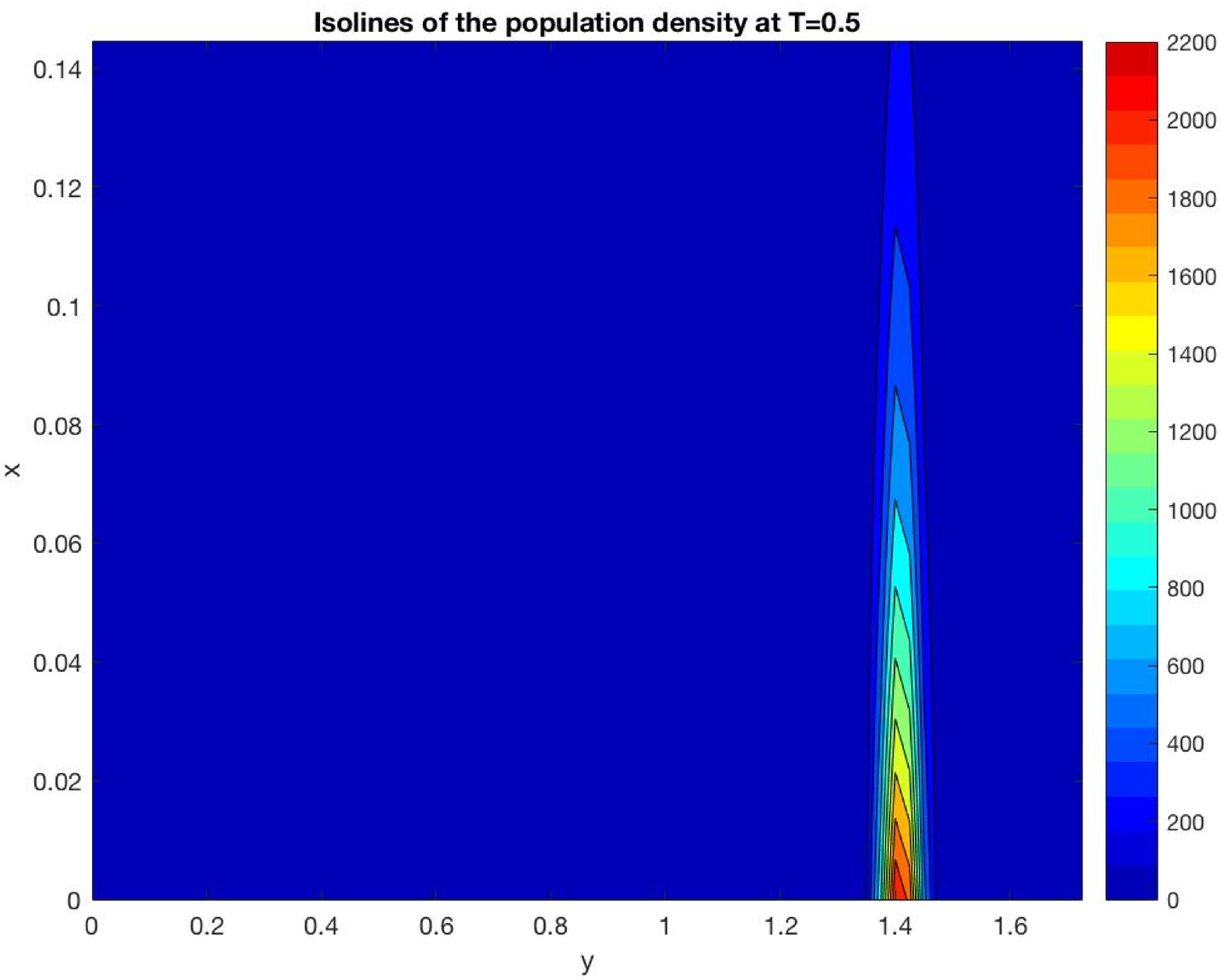}
\vspace{-5mm}
\caption{
Isolines in $(x,y)$ of the population distribution
}
\label{fi.isovaleurs}
\end{figure}

We choose to create a trade-off between the birth and death rates with regards to the $y$ variable, by assuming that $y \mapsto b(x,y)$ and $y \mapsto d(x,y)$ are increasing, which means that a greater natality also induces a greater mortality. This assumption allows to determine an Evolutionary Stable Distribution or ESD from the language of adaptive dynamics, which gives the repartition of the fittest traits (see \cite{WC:PEJ:HL, OD, PEJ:GR}).
 We do not know this ESD from the beginning, however it is important to select, according to assumptions~\eqref{hypothesed}-\eqref{hypM2}, a death rate with a stronger increase for large $x$ than the growth rate with regards to the trait variable in order to avoid that the dominant traits go to infinity.

\begin{figure}[!h]
\centering
\includegraphics[width=0.45\textwidth]{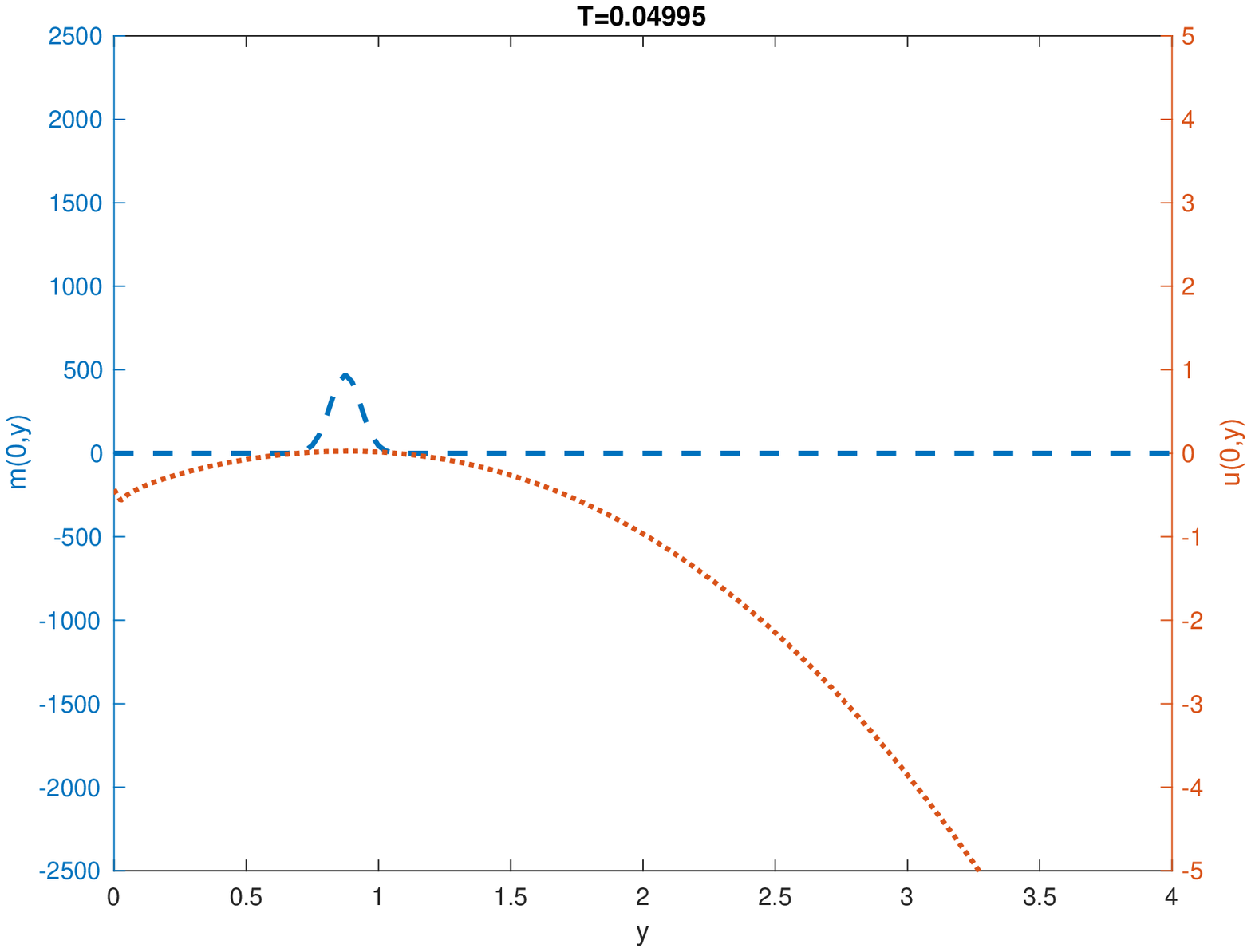}
\includegraphics[width=0.45\textwidth]{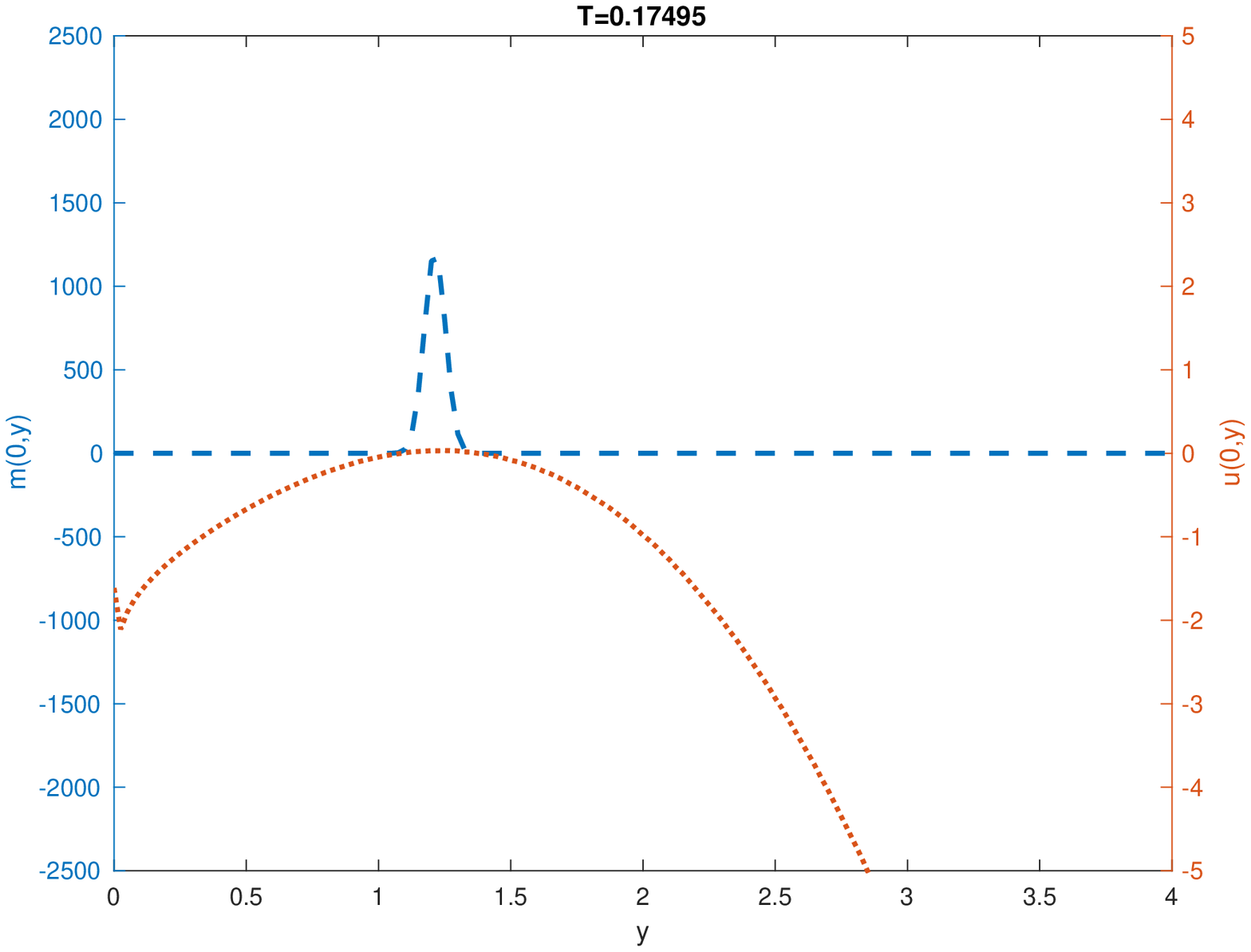}
\includegraphics[width=0.45\textwidth]{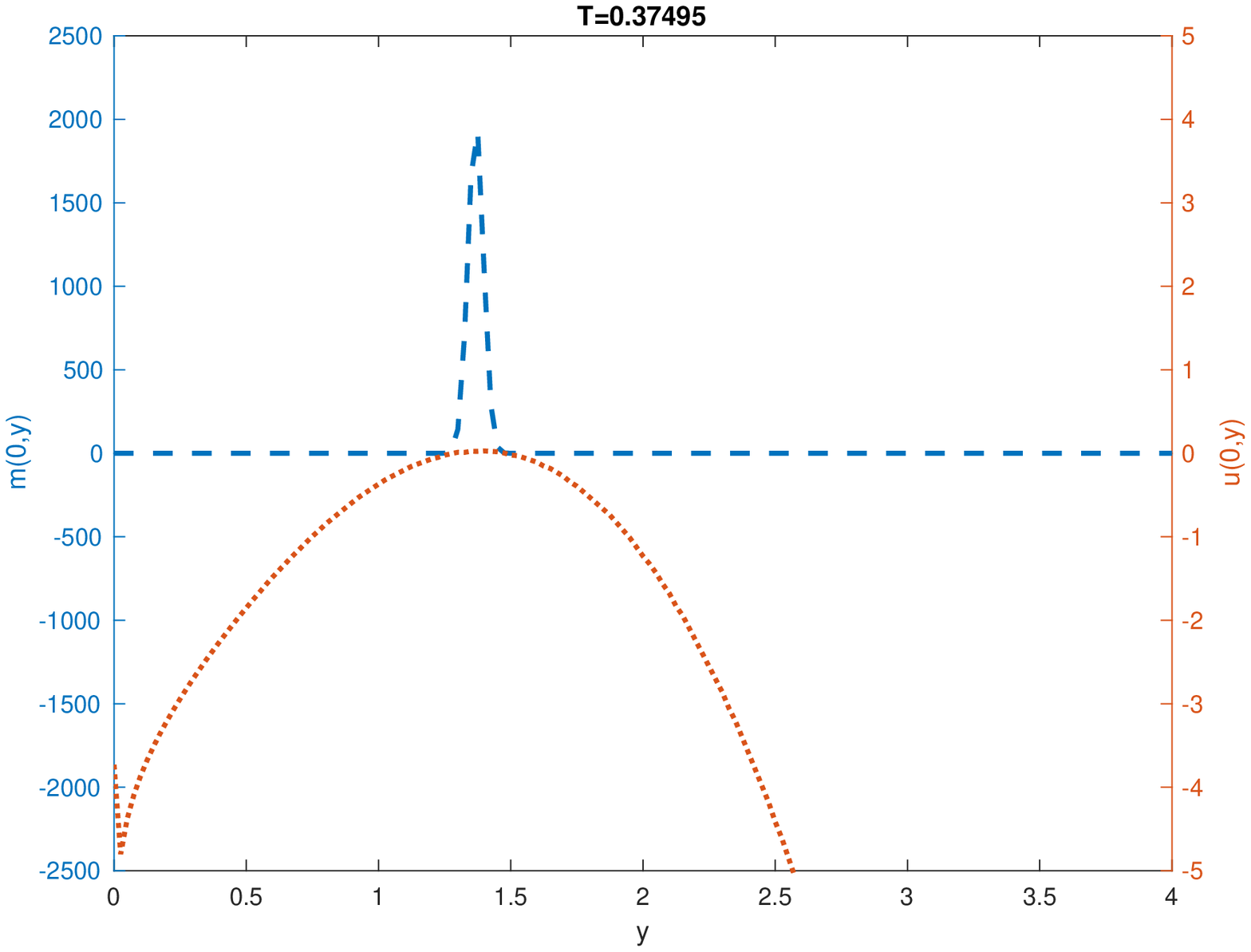}
\includegraphics[width=0.45\textwidth]{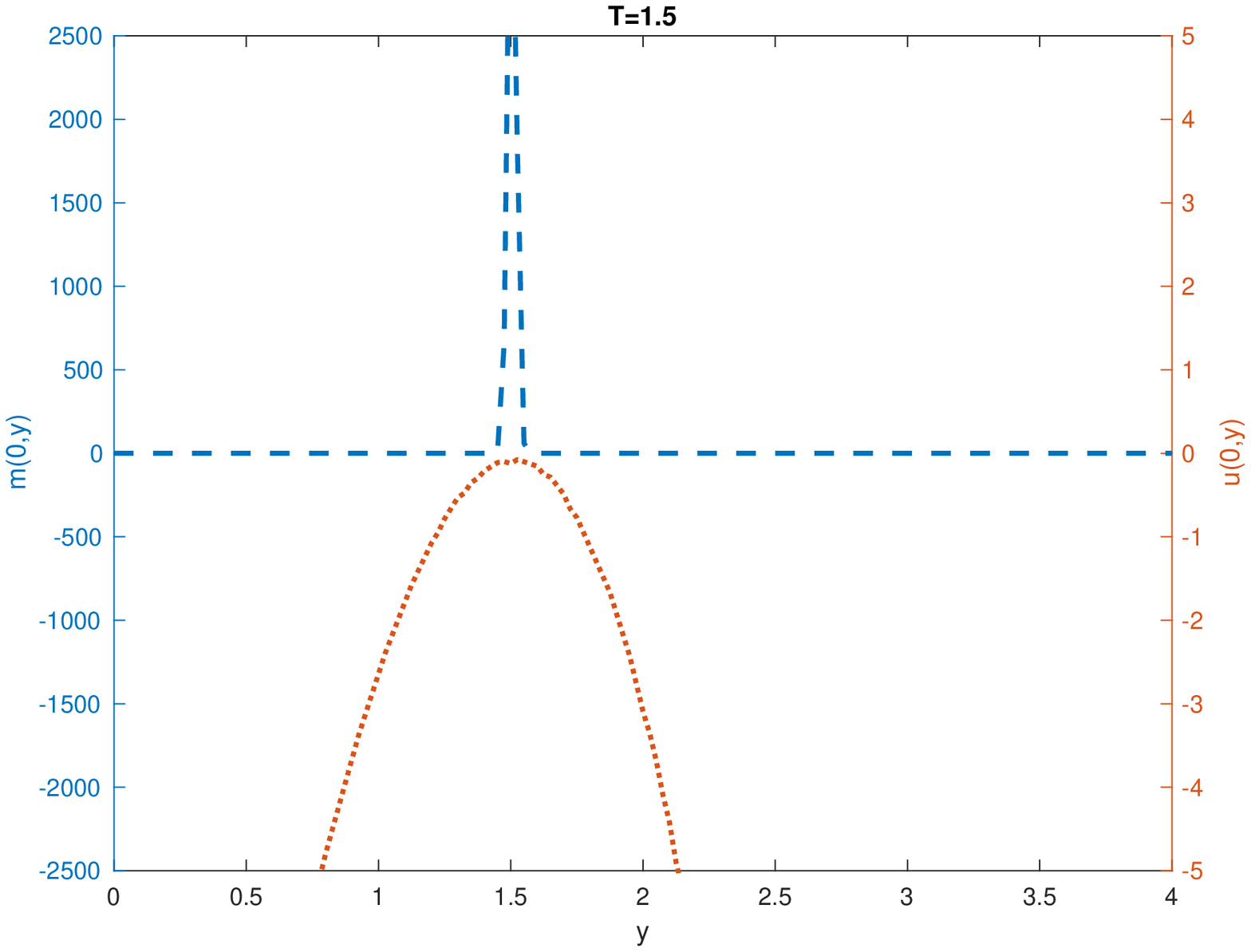}
\vspace{-5mm}
\caption{
Concentration dynamics: snapshots of the population distribution in $y$ at four different times with respect to the trait variable. Blue dashed line= $m_\eps$, red dotted line = $u_\eps$.}
\label{fi.concentration}
\end{figure}
\begin{figure}[!h]
\centering
\includegraphics[width=0.45\textwidth]{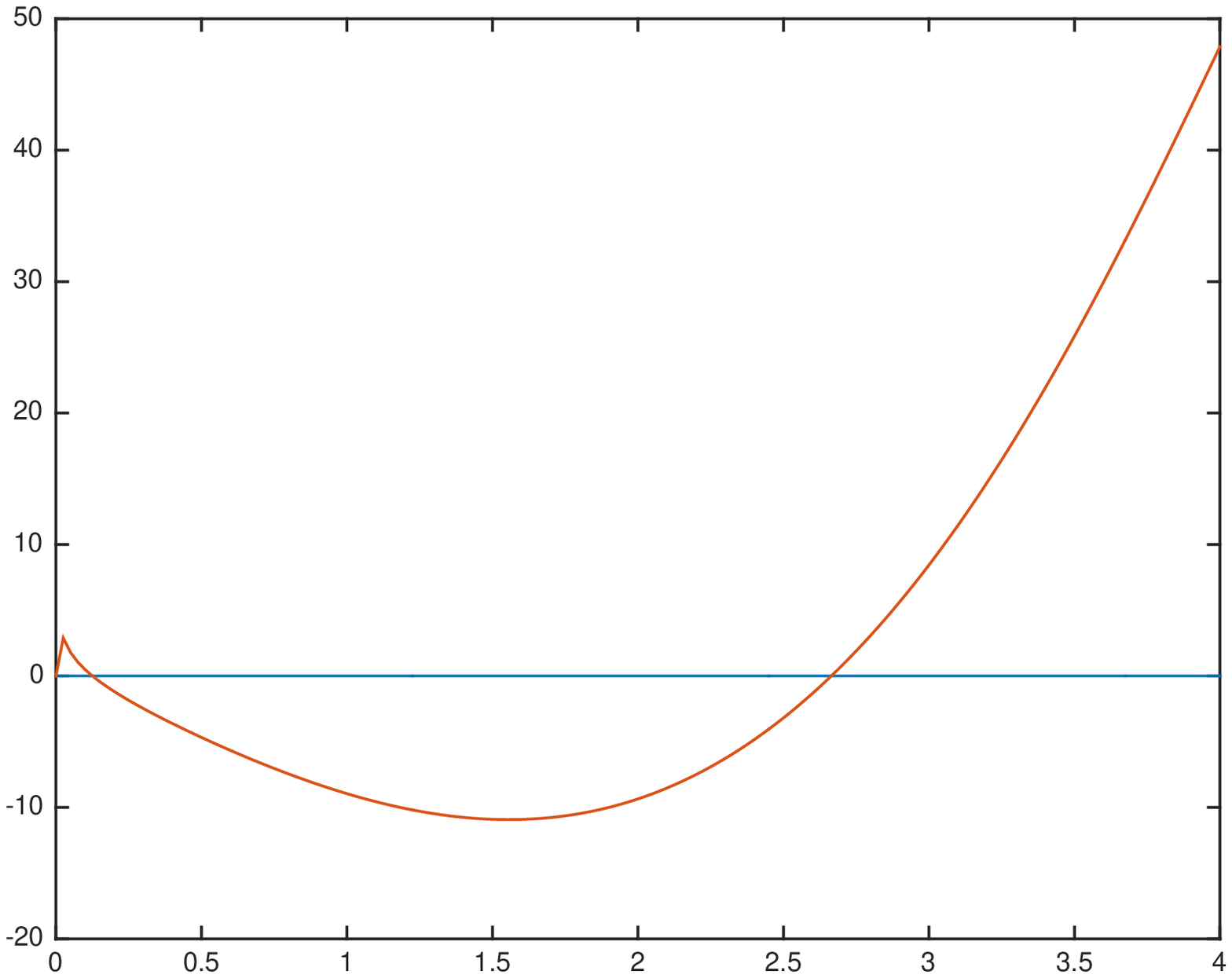}
\includegraphics[width=0.45\textwidth]{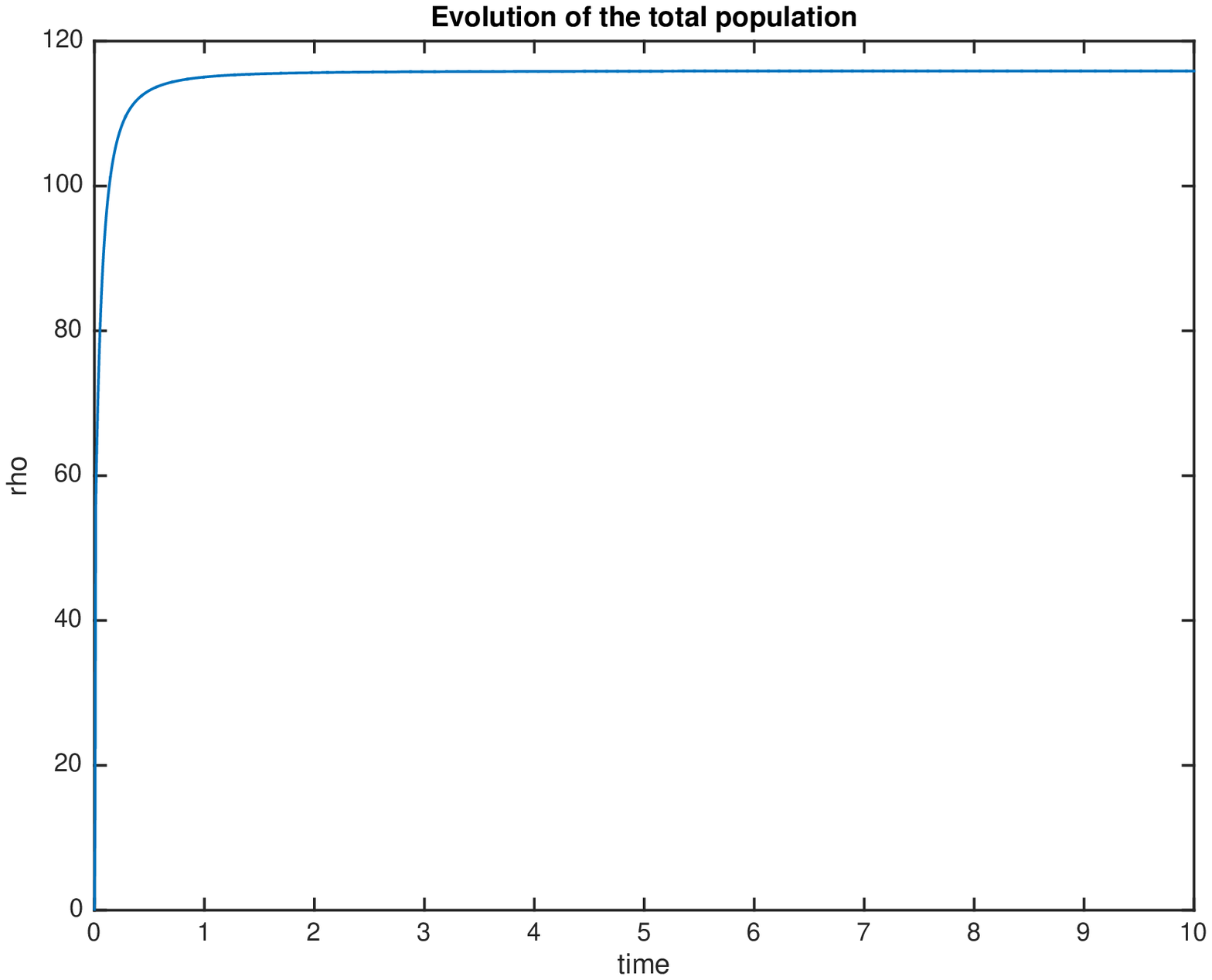}
\caption{Left: Principal eigenvalue $\Lambda(y)$. Right: Evolution of $\rho$ over time}
\label{fi.playR2:ch2}
\end{figure}

Figure~\ref{fi.isovaleurs} shows the population distribution with regards to 
$y$ (abscissa) and $x$ (ordinates) at two different times. The population has 
moved and concentrated to a location which is different from its initial one. 
One can observe this continuous evolution of the population distribution in 
Figure~\ref{fi.concentration} where we show the distribution of individuals with 
age $x=0$ at different times and identify an ESD. 

The ESD can also be identified thanks to the principal eigenvalue. We show in 
Figure~\ref{fi.playR2:ch2} the eigenvalue $\Lambda (y)$ solved by the Newton 
method using~\eqref{derlambdaS}. From equation~\eqref{dynamique} one can notice 
that the equilibrium points have to satisfy $\nabla_y \Lambda(y)=0$ and moreover 
that the dynamics of the concentration is directed towards the minimum points 
of $\Lambda(y)$, as predicted by our analysis.

\section{Conclusion}
The approach we develop here, based on the transformation $m_\epsilon=p_\epsilon e^{\frac{u_\epsilon}{\epsilon}}$, seems convenient for the study concentration phenomena. In the case without mutations, we get precise results on the concentration points as well as on the asymptotic age profile of the population. In particular we have developed a method where the asymptotic analysis is not performed on $u_\eps$ but on $p_\eps$, using relative entropy methods.
Because of technical difficulties, we are not able yet to infer the same conclusion for the case with mutations. However the result seems to hold, at least for short time, more precisely before the Hamilton-Jacobi singularities occur in \eqref{HJU}.  Indeed, denoting 	$Q_\epsilon(t,x,y)= Q(x,y,\eta_\epsilon(t,y))$ we have that $v_\epsilon=\frac{p_\epsilon}{Q_\epsilon}$ satisfies a transport equation with a source term which reads
\begin{equation*}
		\epsilon \partial_t v_\epsilon(t,x,y)+A(x,y)\partial_x v_\epsilon(t,x,y)=\eps\frac{\partial_\eta Q(x,y,\eta_\epsilon(t,y))}{Q(x,y,\eta_\epsilon(t,y))}\partial_t\eta_\epsilon(t,y)v_\epsilon(t,x,y).\\
\end{equation*}
If $\partial_t \eta_\epsilon$ is bounded uniformly, we can deduce that $v_\eps$ is also bounded uniformly, which implies a weak concentration of the population on the set $\{(t,y)\ /\ u(t,y)=0\}$. A rigorous proof of this result along with an entropy method to prove strong convergence of $p_\eps$ will be proposed in a forthcoming paper.


\section{Appendices}

\subsection{Saturation of the population denstity}\label{densityappendix}
We prove Proposition~\ref{saturation} and Proposition \ref{saturation2}. Integrating (\ref{equa}) and using \eqref{hypM2}, we obtain
\begin{equation}\label{eq_rho}
\eps\frac{\de}{\de t}\rho_\eps(t)  =- \int_{\R^n}\int_{\R_+} \big(\D_x [A(x,y)m_\eps(t,x,y)]dxdy+ (d(x,y)+\rho_\eps(t))m_\eps(t,x,y)\big)dxdy.
\end{equation}
First we prove that $A(x,y)m_\eps(t,x,y)$ converges to 0 when $x$ goes to infinity. Note that from \eqref{hypothesed} and the explicit formula for $Q$ given in \eqref{expQ}, we have
\begin{equation*}
\forall y \in \R^n,\quad\lim_{x\to\infty}Q(x,y)= \lim_{x\to\infty}\frac{1}{A(x,y)}\mathrm{exp}\left(-\int_0^{x}{\frac{d(x',y)-\Lambda(y)}{A(x',y)}d x'}\right)=0.
\end{equation*}
Since $p_\eps^0$ is bounded from \eqref{initialp}, we deduce that $m^0_\eps$ converges to 0 when $x$ goes to infinity.  
Besides, as $A$ is bounded and $m_\eps$ satisfies \eqref{equa} which is a transport equation, then a classical result implies that $m_\eps$ converges to 0 when $x$ goes to infinity.

Then, integrating by parts in \eqref{eq_rho}, we obtain
\begin{align*}
\eps\frac{\de}{\de t}\rho_\eps(t) &=\int_{\R^n}\int_{\R_+}\left[\left(\frac{1}{\eps^n}\int_{\R^n}{M(\frac{y'-y}{\eps})d y}\right)b(x,y') -d(x,y')\right]m_\eps(t,x,y')dxdy' -\rho_\eps^2(t)\\
 &\displaystyle\leq \overline{r}\rho_\eps(t) -\rho_\eps^2(t). \phantom{\int}
\end{align*}
Therefore, using (\ref{hypM1bis}), we conclude
$$0\leq\rho_\eps(t)\leq\max\left({\overline{r}},\rho_\eps^0\right).$$
The other inequality can be proved in the same way.

\subsection{Proof of Theorem \ref{theoremeeigen} and Theorem 
\ref{theoremeeigenS}}\label{eigenappendix}
We only prove Theorem \ref{theoremeeigen}, as Theorem~\ref{theoremeeigenS} is a particular case with $\eta=1$. Equation~(\ref{valeurpropre}) is equivalent to write
\begin{equation*}
Q(x,y,\eta)=Q(0,y,\eta)\mathrm{exp}\left(-\int_0^{x}{\frac{\D_xA(x',y)+d(x',y)-\Lambda(y,\eta)}{A(x',y)}dx'}\right),\end{equation*} 
and thanks to the condition at $x=0$,
	\begin{equation}\label{DefinitionQ}
		Q(x,y,\eta)= \eta\frac{1}{A(x,y)}\mathrm{exp}\left(-\int_0^{x}{\frac{d(x',y)-\Lambda(y,\eta)}{A(x',y)}d x'}\right).
	\end{equation}
Multiplying by $b(x,y)$ and integrating with regard to the $x$ variable, we obtain 
\begin{equation}\label{FormuleImplicite}
\frac{1}{\eta}=F(y,\Lambda(y,\eta)).
\end{equation}
A direct calculation gives $\D_\lambda F>0$, thus \eqref{FormuleImplicite} ensures uniqueness for $\Lambda$ and then for $Q$.

Moreover, as $F(y,+\infty)=+\infty$ and $F(y,-\infty)=0$, there exists such a $\Lambda(y,\eta)$. Besides, defining $Q$ as in~\eqref{DefinitionQ} implies that $Q$ is in $L^1\cap L^\infty$, thanks to~(\ref{hypothesed}), thus it proves existence.
Finally, using the implicit function theorem in~\eqref{FormuleImplicite} we deduce that $\Lambda(y,\eta)$ is $\mathcal{C}^1$ and~(\ref{detlambda}) holds true. 

For the dual equation \eqref{EquationDualeBis}, a simple calculation shows that the solution $\Phi$ must be given by
\begin{equation}
\Phi(x,y,\eta)=\Phi(0,y,\eta)e^{-\int_0^x \frac{\Lambda(y,\eta)-d(x',y)}{A(x',y)}d x'}\left(1-\eta\int_0^x\frac{b(x',y)}{A(x',y)}e^{\int_0^{x'}\frac{\Lambda(y,\eta)-d(x'',y)}{A(x'',y)}d x''}\right),
\end{equation}
where $\Phi(0,y,\eta)>0$ is determined by the normalization $\int_{\R_+}{Q(x,y,\eta)\Phi(x,y,\eta)d x}=1.$

Finally, we prove in the case without mutations
\begin{equation}
\forall y \in \R^n, \quad \underline{r} \leq -\Lambda(y)\leq \overline{r}.
\end{equation}
Integrating \eqref{valeurpropreS} with respect to $x$, we have
\begin{equation}
-\Lambda(y)=\frac{\int_{\R_+}(b(x,y)-d(x,y))Q(x,y)dx}{\int_{\R_+}Q(x,y)dx}.
\end{equation}
Thus, using \eqref{hypM2}, we obtain the announced result.

\subsection{Proof of Proposition~\ref{PropConvexity}}\label{HConvexity}

	We first state the following lemma. We recall that the definitions of $F(y,\lambda), \Lambda(y,\eta)$ and $\eta(p)$ are given in~\eqref{definitionF}, \eqref{deriveelambda} and \eqref{DefinitionHamiltonian}.
	\begin{Lemme}\label{lemmemajoration}
		We have 
		\begin{equation}\label{eq:majorationderive}
				\eta(p) \big[\D_\lambda F\big(y,\Lambda(y,\eta(p))\big) \big]^2\leq\D^2_{\lambda}F\big(y,\Lambda(y,\eta(p))\big),
		\end{equation}
		and
		\begin{equation}\label{eq:majorationderiveeta}
			\big[ \D_{p_i}\eta(p)\big]^2\leq\eta(p)\D^2_{\p}\eta(p).
		\end{equation}
	\end{Lemme}
	\begin{proof}[Proof of Lemma \ref{lemmemajoration}\newline]
	
	We define and compute using \eqref{definitionF} 
\begin{equation*}
g(x,y):=\int_0^x\frac{1}{A(x',y)}\de x', \quad
\D_\lambda F(y,\lambda)=\int_0^\infty g(x,y)f(x,y,\lambda)dx,
\end{equation*}
With these notations we may write
\begin{equation*}
\D^2_{\lambda}F(y,\lambda)=\int_0^\infty g(x,y)^2f(x,y,\lambda)dx.
\end{equation*}
	Using the Cauchy-Schwarz inequality we obtain
	\begin{equation*}
		\big[\D_\lambda F \big(y,\Lambda(y,\eta(p)) \big)\big]^2\leq\D^2_{\lambda}F \big(y,\Lambda(y,\eta(p)) \big)\cdot F \big(y,\Lambda(y,\eta(p)) \big),
	\end{equation*}
	and then thanks to \eqref{deriveelambda} the first inequality follows.
	The second inequality is a simple consequence of the Cauchy-Schwarz inequality on 
	$
	\eta(p)=\int_{\R^n}M(z)e^{p\cdot z}dz.
	$
\end{proof}

We go back to the proof of Proposition~\ref{PropConvexity}. By differentiating twice \eqref{deriveelambda} with respect to $p_i$, we obtain
	\begin{equation}\label{derivepremiere}
		\D_\lambda F\big(y,\Lambda(y,\eta(p))\big)D_{p_i}\Lambda(y, \eta (p))=-\frac{\D_{p_i}\eta(p)}{{\eta(p)}^2},
	\end{equation}
\begin{equation*}
\D_\lambda F \cdot D^2_{p_i}\Lambda(y,\eta(p))+\D^2_{\lambda}F \cdot\big[D_{p_i}\Lambda(y,\eta(p))\big]^2=-\frac{\D_{p_i^2}\eta(p)}{\eta(p)^2}+2\frac{\D_{p_i}\eta(p)}{\eta(p)^3}.	
\end{equation*}
Then using \eqref{eq:majorationderive}, \eqref{eq:majorationderiveeta} and \eqref{derivepremiere}, we derive
	\begin{equation*}
		\begin{aligned}
			\D_\lambda F \cdot D^2_{p_i}\Lambda(y,p)&=-\D^2_{\lambda}F \left[\frac{\D_{p_i}\eta(p)}{ \eta(p)^2\D_\lambda F}\right]^2 -\frac{\D_{p_i^2}\eta(p)}{\eta(p)^2}+2\frac{\big[\D_{p_i}\eta(p)\big]^2}{\eta(p)^3}\\
			&\leq -\frac{\big[\D_{p_i}\eta(p)\big]^2}{\eta(p)^3}-\frac{\D_{p_i^2}\eta(p)}{\eta(p)^2}+2\frac{\big[\D_{p_i}\eta(p)\big]^2}{\eta(p)^3}\\
&= -\frac{1}{\eta(p)^3}\left(\eta(p)\D^2_{p_i}\eta(p)-\big[\D_{p_i}
\eta(p)\big]^2\right) \leq 0,	
\end{aligned}
	\end{equation*}
hence the announced convexity result on $p\mapsto H(y,p)$.

\subsection{Proof of Proposition~\ref{borneV}}\label{AppAPrioriEstimate}

Our goal is to prove
\begin{equation}\label{objectif}
\D_tU_\eps^R(t,y)\leq\sup_{y\in \R^d}\D_tU^{R,0}_\eps:=\sup_{y\in \R^n}\D_t u^0_\eps(0,y),\quad \forall R>0, \forall y \in \R^n, \forall t>0.
\end{equation}
The reverse inequality can be obtained similarly. Note that from \eqref{H1} we have that
	\begin{equation*}
		\D_tU^{0,R}_\eps=-\Lambda\left(y,\int_{\R^n}{M(z)e^{\frac{u^0_\eps(y+\eps z)-u^0_\eps(y)}{\eps}}d z}\right)\text{ is bounded uniformly in }\eps, 
	\end{equation*}
thus \eqref{objectif} allows us to conclude that $\D_tU^{R}_\eps$ is bounded uniformly in $R$ and $\eps$.

\paragraph{}
	
We prove \eqref{objectif} by contradiction. We assume that there exists $(T,y_0)\in(0,+\infty)\times\R^n$ such that
	\begin{equation}\label{absurd2}
		\D_tU_\eps^R(T,y_0)-\sup\D_tU_\eps^{R,0}>0.
	\end{equation}
For conciseness, we define
$
		V_\eps^R(t,y):=\D_t U_\eps^R(t,y).
$
	For $\beta>0$, $\alpha>0$ small and for $t\in[0,T], y\in\R^n,$ we also introduce
	\begin{equation*}
			\varphi_{\alpha,\beta}(t,y):=V_\eps^R(t,y)- \alpha t-\beta\vert y-y_0\vert.
	\end{equation*}
	We choose $\alpha$ small enough to ensure $\varphi_{\alpha,\beta}(T,y_0)>\varphi_{\alpha,\beta}(0,y_0)=\D_tU_\eps^{R,0}(y_0),$ which is possible thanks to assumption \eqref{absurd2}.
From the definition of $\phi_R$, we have $\vert V^R_\eps(t,y)\vert\leq R$, therefore $\varphi_{\alpha,\beta}$ decreases to $-\infty$ as $\vert y\vert \to \infty$ and reaches its maximum on ${[0,T]\times\R^n}$ at a point $(\bar t, \bar y)$. We have
\begin{equation*}
	\varphi_{\alpha,\beta}(\bar t,\bar y+\eps z)\leq\varphi_{\alpha,\beta}(\bar t,\bar y),\quad\forall z\in\R^n,
\end{equation*}
and thus
\begin{equation}\label{accroissement1}
	\frac{V^R_\eps(\bar t,\bar y+\eps z)-V^R_\eps(\bar t,\bar y)}{\eps}\leq\beta\frac{\vert\bar y+\eps z\vert-\vert\bar y\vert}{\eps}\leq\beta \vert z\vert,\quad \forall z\in\R^n.
\end{equation}
Moreover, as $u^0_\eps$ is $k_0$-Lipschitz continuous from \eqref{initialu}, then	we obtain for all $t>0,\ (y,y')\in\R^{2n},$
\begin{multline}\label{lipschitz2}
\vert U_\eps^R(t,y)-U_\eps^R(t,y')\vert \\
\leq \vert 
U_\eps^R(t,y)-U_\eps^{0,R}(y)\vert+ \vert 
U_\eps^{0,R}(y)-U_\eps^{0,R}(y')\vert+\vert 
U_\eps^{0,R}(y')-U_\eps^R(t,y')\vert\\
\leq 2RT+k_0\vert y-y'\vert.
\end{multline}
Next, we set 
	\begin{gather}
		\eta_\eps^R(t,y):=\int_{\R^n}{M(z)e^{\frac{U^R_\eps(t,y+\eps z)-U^R_\eps(t,y)}{\eps}}}d z, \\
	\eta_\eps^{\pm}:=\int_{\R^n}{M(z)e^{\pm\left(\frac{2RT}{\eps}+k_0\vert z\vert\right)}}dz,		
	\end{gather}
and notice that $0<\eta_\eps^-\leq\eta_\eps^R(t,y)\leq\eta_\eps^+.$

We have chosen $\alpha$ such that $\varphi_{\alpha,\beta}(0,y_0)<\varphi_{\alpha,\beta}(T,y_0)$, then we know that $\bar t>0$. Hence ${\D_t\varphi_{\alpha,\beta}(\bar t,\bar y)\geq0}$, that is ${\D_tV_\eps^R(\bar t,\bar y)\geq\alpha}$ (if $\bar t=T$ then $\D_t V_\eps^R(\bar t,\bar y)$ stands for the left-derivative).
Differentiating \eqref{eqR}, we have
\begin{equation}
\label{eqRD1}
		\D_tV_\eps^R(t,y)=\phi_R'\left(-\Lambda\left(y,\eta_\eps^R\right)\right)\left(-\D_\eta\Lambda\left(y,\eta_\eps^R\right)\right)\Gamma_\eps^R(t,y),
\end{equation}
where $\Gamma_\eps^R(t,y):=\int_{\R^n}{M(z)e^{\frac{U^R_\eps(t,y+\eps z)-U^R_\eps(t,y)}{\eps}}\left(\frac{V^R_\eps(t,y+\eps z)-V^R_\eps(t,y)}{\eps}\right)dz}$.

\paragraph{}

Writing \eqref{eqRD1} at $(\bar t,\bar y)$, using \eqref{detlambda} and 
\eqref{accroissement1}-\eqref{lipschitz2}, we have
\begin{equation}
\begin{aligned}
\alpha&\displaystyle\leq\D_tV_\eps^R(\bar t,\bar y)= 
\phi_R'\left(-\Lambda\left(y,\eta_\eps^R(\bar t,\bar 
y)\right)\right)\left(-\D_\eta\Lambda\left(y,\eta_\eps^R(\bar t,\bar 
y)\right)\right)\Gamma_\eps^R(\bar t,\bar y) \phantom{\int}\\
&\leq\beta\sup_{r \in 
\R}\phi_R'(r)\sup_{\substack{\eta\in(\eta_\eps^-,\eta_\eps^+)\\y\in\R^n}}\left[
-\D_\eta\Lambda\left(y,\eta\right)\right]\left(\int_{\R^n}{M(z)e^{\frac{
U^R_\eps(\bar t,\bar y+\eps z)-U^R_\eps(\bar t,\bar y)}{\eps}}\vert z\vert d 
z}\right) \\
&\leq\beta\sup_{r \in 
\R}\phi_R'(r)\sup_{\substack{\eta\in(\eta_\eps^-,\eta_\eps^+)\\y\in\R^n}}\left[
-\D_\eta\Lambda\left(y,\eta\right)\right]\left(\int_{\R^n}{M(z)e^{\frac{2RT}{
\eps}+k_0\vert z\vert}\vert z\vert dz}\right).
\end{aligned}
\end{equation}
Hence $\alpha\leq \bar C\beta$, where $\bar C$ is a constant that does not depend on $\beta$. Then as $\beta$ goes to $0$, we obtain $\alpha\leq0$, which is absurd. The proof is thereby achieved.

%
%

\bibliographystyle{plain}
\bibliography{bibli.bib}

\begin{thebibliography}{10}

\bibitem{ackleh_F_T}
Azmy~S. Ackleh, Ben~G. Fitzpatrick, and Horst~R. Thieme.
\newblock {Rate distributions and survival of the fittest: a formulation on the
  space of measures}.
\newblock {\em Discrete Contin. Dyn. Syst. Ser. B}, 5(4):917--928, 2005.

\bibitem{Adimy-Cr-Ru}
Mostafa Adimy, Fabien Crauste, and Shigui Ruan.
\newblock {A mathematical study of the hematopoiesis process with applications
  to chronic myelogenous leukemia}.
\newblock {\em SIAM J. Appl. Math.}, 65(4):1328--1352, 2005.

\bibitem{bardi-capuzzo}
Martino Bardi and Italo Capuzzo-Dolcetta.
\newblock {\em {Optimal control and viscosity solutions of
  {H}amilton-{J}acobi-{B}ellman equations}}.
\newblock {Systems \& Control: Foundations \& Applications}. Birkh{\"a}user
  Boston, Inc., Boston, MA, 1997.
\newblock With appendices by Maurizio Falcone and Pierpaolo Soravia.

\bibitem{BG-LCE-PES}
G.~Barles, L.~C. Evans, and P.~E. Souganidis.
\newblock {Wavefront propagation for reaction-diffusion systems of {PDE}}.
\newblock {\em Duke Math. J.}, 61(3):835--858, 1990.

\bibitem{combustion}
G.~Barles, C.~Georgelin, and P.~E. Souganidis.
\newblock {Front propagation for reaction-diffusion equations arising in
  combustion theory}.
\newblock {\em Asymptot. Anal.}, 14(3):277--292, 1997.

\bibitem{GB.BP:88}
G.~Barles and B.~Perthame.
\newblock {Exit time problems in optimal control and vanishing viscosity
  method}.
\newblock {\em SIAM J. Control Optim.}, 26(5):1133--1148, 1988.

\bibitem{GB:94}
Guy Barles.
\newblock {\em {Solutions de viscosit{\'e} des {\'e}quations de
  {H}amilton-{J}acobi}}, volume~17 of {\em {Math{\'e}matiques \& Applications
  (Berlin) [Mathematics \& Applications]}}.
\newblock Springer-Verlag, Paris, 1994.

\bibitem{GB.SM.BP:09}
Guy Barles, Sepideh Mirrahimi, and Beno{\^\i}t Perthame.
\newblock {Concentration in {L}otka-{V}olterra parabolic or integral equations:
  a general convergence result}.
\newblock {\em Methods Appl. Anal.}, 16(3):321--340, 2009.

\bibitem{Ba-Pe}
Guy Barles and Beno{\^\i}t Perthame.
\newblock {Concentrations and constrained {H}amilton-{J}acobi equations arising
  in adaptive dynamics}.
\newblock In {\em {Recent developments in nonlinear partial differential
  equations}}, volume 439 of {\em {Contemp. Math.}}, pages 57--68. Amer. Math.
  Soc., Providence, RI, 2007.

\bibitem{Busse2016}
J.-E. Busse, P.~Gwiazda, and A.~Marciniak-Czochra.
\newblock {Mass concentration in a nonlocal model of clonal selection}.
\newblock {\em J. Math. Biol.}, 73(4):1001--1033, 2016.

\bibitem{WC:PEJ:HL}
Wenli Cai, Pierre-Emmanuel Jabin, and Hailiang Liu.
\newblock {Time-asymptotic convergence rates towards the discrete evolutionary
  stable distribution}.
\newblock {\em Math. Models Methods Appl. Sci.}, 25(8):1589--1616, 2015.

\bibitem{calsina-al-2013}
{\`A}ngel Calsina, S{\'i}lvia Cuadrado, Laurent Desvillettes, and Ga{\"e}l
  Raoul.
\newblock {Asymptotics of steady states of a selection-mutation equation for
  small mutation rate}.
\newblock {\em Proc. Roy. Soc. Edinburgh Sect. A}, 143(6):1123--1146, 2013.

\bibitem{calsina-2013}
{\`A}ngel Calsina and Josep~M. Palmada.
\newblock {Steady states of a selection-mutation model for an age structured
  population}.
\newblock {\em J. Math. Anal. Appl.}, 400(2):386--395, 2013.

\bibitem{VC:PG:AMG}
V.~{Calvez}, P.~{Gabriel}, and {\'A}.~{Mateos Gonz{\'a}lez}.
\newblock {Limiting Hamilton-Jacobi equation for the large scale asymptotics of
  a subdiffusion jump-renewal equation}.
\newblock {\em ArXiv e-prints}, September 2016.

\bibitem{CFBA}
N.~Champagnat, R.~Ferri{\`e}re, and G.~{Ben Arous}.
\newblock {The canonical equation of adaptive dynamics: A mathematical view}.
\newblock {\em Selection}, 2(1-2):73--83, 2002.

\bibitem{champagnat2011}
Nicolas Champagnat and Pierre-Emmanuel Jabin.
\newblock {The evolutionary limit for models of populations interacting
  competitively via several resources}.
\newblock {\em J. Differential Equations}, 251(1):176--195, 2011.

\bibitem{Cr-Ev-Li}
M.~G. Crandall, L.~C. Evans, and P.-L. Lions.
\newblock {Some properties of viscosity solutions of {H}amilton-{J}acobi
  equations}.
\newblock {\em Trans. Amer. Math. Soc.}, 282(2):487--502, 1984.

\bibitem{crandall1992}
Michael~G. Crandall, Hitoshi Ishii, and Pierre-Louis Lions.
\newblock {User's guide to viscosity solutions of second order partial
  differential equations}.
\newblock {\em Bull. Amer. Math. Soc. (N.S.)}, 27(1):1--67, 1992.

\bibitem{Des-Jab-Mis-Rao}
Laurent Desvillettes, Pierre-Emmanuel Jabin, St{\'e}phane Mischler, and
  Ga{\"e}l Raoul.
\newblock {On selection dynamics for continuous structured populations}.
\newblock {\em Commun. Math. Sci.}, 6(3):729--747, 2008.

\bibitem{Di-Ja-Mi-Pe}
O.~Diekmann, P.-E. Jabin, S.~Mischler, and B.~Perthame.
\newblock {The dynamics of adaptation: an illuminating example and a
  Hamilton-Jacobi approach}.
\newblock {\em Theor. Popul. Biol.}, 67(4):257--271, 2005.

\bibitem{OD}
Odo Diekmann.
\newblock {A beginner's guide to adaptive dynamics}.
\newblock In {\em {Mathematical modelling of population dynamics}}, volume~63
  of {\em {Banach Center Publ.}}, pages 47--86. Polish Acad. Sci. Inst. Math.,
  Warsaw, 2004.

\bibitem{Dji-Duc-Fab}
Ramses Djidjou-Demasse, Arnaud Ducrot, and Fr{\'e}d{\'e}ric Fabre.
\newblock {Steady state concentration for a phenotypic structured problem
  modeling the evolutionary epidemiology of spore producing pathogens}.
\newblock {\em Math. Models Methods Appl. Sci.}, 27(2):385--426, 2017.

\bibitem{Doumic-Gabriel}
Marie {Doumic Jauffret} and Pierre Gabriel.
\newblock {Eigenelements of a general aggregation-fragmentation model}.
\newblock {\em Math. Models Methods Appl. Sci.}, 20(5):757--783, 2010.

\bibitem{Eva}
L.~C. Evans and P.~E. Souganidis.
\newblock {A {PDE} approach to geometric optics for certain semilinear
  parabolic equations}.
\newblock {\em Indiana Univ. Math. J.}, 38(1):141--172, 1989.

\bibitem{fle-soner}
Wendell~H. Fleming and H.~Mete Soner.
\newblock {\em {Controlled {M}arkov processes and viscosity solutions}},
  volume~25 of {\em {Applications of Mathematics (New York)}}.
\newblock Springer-Verlag, New York, 1993.

\bibitem{GyllWebb}
M.~Gyllenberg and G.~F. Webb.
\newblock {A nonlinear structured population model of tumor growth with
  quiescence}.
\newblock {\em J. Math. Biol.}, 28(6):671--694, 1990.

\bibitem{HI}
Hitoshi Ishii.
\newblock {Hamilton-{J}acobi equations with discontinuous {H}amiltonians on
  arbitrary open sets}.
\newblock {\em Bull. Fac. Sci. Engrg. Chuo Univ.}, 28:33--77, 1985.

\bibitem{PEJ:RSS}
P.-E. {Jabin} and R.~{Strother Schram}.
\newblock {Selection-Mutation dynamics with spatial dependence}.
\newblock {\em ArXiv e-prints}, January 2016.

\bibitem{PEJ:GR}
Pierre-Emmanuel Jabin and Ga{\"e}l Raoul.
\newblock {On selection dynamics for competitive interactions}.
\newblock {\em J. Math. Biol.}, 63(3):493--517, 2011.

\bibitem{lorenzi-2013}
Tommaso Lorenzi, Alexander Lorz, and Giorgio Restori.
\newblock {Asymptotic dynamics in populations structured by sensitivity to
  global warming and habitat shrinking}.
\newblock {\em Acta Appl. Math.}, 131:49--67, 2014.

\bibitem{AL.SM.BP}
Alexander Lorz, Sepideh Mirrahimi, and Beno{\^\i}t Perthame.
\newblock {Dirac mass dynamics in multidimensional nonlocal parabolic
  equations}.
\newblock {\em Comm. Partial Differential Equations}, 36(6):1071--1098, 2011.

\bibitem{DiekPhys}
J.~A.~J. Metz and O.~Diekmann, editors.
\newblock {\em {The dynamics of physiologically structured populations}},
  volume~68 of {\em {Lecture Notes in Biomathematics}}.
\newblock Springer-Verlag, Berlin, 1986.
\newblock Papers from the colloquium held in Amsterdam, 1983.

\bibitem{PMichel}
Philippe Michel.
\newblock {Existence of a solution to the cell division eigenproblem}.
\newblock {\em Math. Models Methods Appl. Sci.}, 16(7, suppl.):1125--1153,
  2006.

\bibitem{PMSMBP1}
Philippe Michel, St{\'e}phane Mischler, and Beno{\^\i}t Perthame.
\newblock {General relative entropy inequality: an illustration on growth
  models}.
\newblock {\em J. Math. Pures Appl. (9)}, 84(9):1235--1260, 2005.

\bibitem{Mi-patch}
Sepideh Mirrahimi.
\newblock {Adaptation and migration of a population between patches}.
\newblock {\em Discrete Contin. Dyn. Syst. Ser. B}, 18(3):753--768, 2013.

\bibitem{Mi-Pe-spatial}
Sepideh Mirrahimi and Beno{\^\i}t Perthame.
\newblock {Asymptotic analysis of a selection model with space}.
\newblock {\em J. Math. Pures Appl. (9)}, 104(6):1108--1118, 2015.

\bibitem{Mis-Pe-Ry}
St{\'e}phane Mischler, Beno{\^\i}t Perthame, and Lenya Ryzhik.
\newblock {Stability in a nonlinear population maturation model}.
\newblock {\em Math. Models Methods Appl. Sci.}, 12(12):1751--1772, 2002.

\bibitem{BP:PS-dispersal}
B.~Perthame and P.~E. Souganidis.
\newblock {Rare mutations limit of a steady state dispersal evolution model}.
\newblock {\em Math. Model. Nat. Phenom.}, 11(4):154--166, 2016.

\bibitem{BP}
Beno{\^\i}t Perthame.
\newblock {\em {Transport equations in biology}}.
\newblock {Frontiers in Mathematics}. Birkh{\"a}user Verlag, Basel, 2007.

\bibitem{GB.BP:07}
Beno{\^\i}t Perthame and Guy Barles.
\newblock {Dirac concentrations in {L}otka-{V}olterra parabolic {PDE}s}.
\newblock {\em Indiana Univ. Math. J.}, 57(7):3275--3301, 2008.

\end{thebibliography}

\end{document}